\documentclass[12pt]{article}
\pdfoutput=1
\usepackage{graphicx}
\usepackage{mathptmx}
\usepackage{mathrsfs}
\usepackage{amssymb}
\usepackage{amsmath}
\usepackage{graphicx}
\usepackage{amsmath,amsfonts}
\usepackage{algorithm}
\usepackage{algorithmic}
\usepackage[round]{natbib}
\usepackage{amscd}
\usepackage[mathscr]{eucal}

\usepackage{tikz}
\newcommand*\circledchar[1]{
\begin{tikzpicture}
\node[draw,circle,inner sep=0.7pt] {#1};
\end{tikzpicture}}

\newcommand{\uQ}{\mathbb{Q}}
\newcommand{\uR}{\mathbb{R}}

\newcommand{\uN}{\mathbb{N}}
\newcommand{\uZ}{\mathbb{Z}}
\newcommand{\uA}{\mathcal{A}}

\newcommand{\M}{\mathcal{M}}

\newcommand{\opQ}{\mathbf{Q}}

\newtheorem{thm}{Theorem}[section]
\numberwithin{thm}{section}
\newtheorem{exam}[thm]{Example}%[section]
\newtheorem{rem}[thm]{Remark}%[section]
\newtheorem{defn}[thm]{Definition}%[section]
\newtheorem{cor}[thm]{Corollary}%[section]
\newtheorem{prop}[thm]{Proposition}%[section]
\newtheorem{assm}[thm]{Assumption}%[section]

\newenvironment{proof}{\noindent\\ \noindent\relax{\sc
     Proof}}{{\samepage\par\nopagebreak\hbox
     to\hsize{\hfill$\Box$}}}
\newcommand{\be}{\begin{equation}} \newcommand{\ee}{\end{equation}}
\newcommand{\bd}{\begin{displaymath}} \newcommand{\ed}{\end{displaymath}}
\newcommand{\ba}{\begin{align}} \newcommand{\ea}{\end{align}}
\newcommand{\baa}{\begin{align*}} \newcommand{\eaa}{\end{align*}}
\newcommand{\ben}{\begin{enumerate}} \newcommand{\een}{\end{enumerate}}
\newcommand{\bi}{\begin{itemize}} \newcommand{\ei}{\end{itemize}}

\newcommand{\uP}{\mathcal{P}}
\newcommand{\ud}{\mathrm{d}}

\begin{document}

\title{Centred quadratic stochastic operators}
\author{{\sc Krzysztof Bartoszek} and {\sc Joachim Domsta} and {\sc Ma\l gorzata Pu\l ka}} 

\maketitle

\begin{abstract}
We study the weak convergence of iterates of so--called centred kernel quadratic
stochastic operators. These iterations, in a population evolution setting,
describe the additive perturbation of the arithmetic
mean of the traits of an individual's parents
and correspond to certain weighted sums of independent random 
variables. We show that one can obtain weak convergence results 
under rather mild assumptions on the kernel.
Essentially it is sufficient for the 
distribution of the perturbing random variable 
to have a finite variance or have tails controlled by a power function.
The advantage of these
conditions is that in many cases they 
are
easily verifiable by an applied user. 
Additionally, the representation by sums of random variables implies
an efficient simulation algorithm
to obtain random variables approximately following 
the law of the iterates of the quadratic stochastic operator, 
with full control of the degree of approximation.
Our results also indicate where lies an intrinsic difficulty
in the analysis of the behaviour of quadratic stochastic operators.
\end{abstract}


\section{Introduction}\label{intro}
The theory of quadratic stochastic operators (QSOs) is rooted in the work of \citet{SBer1924}.
He applied such operators to model the evolution of a discrete probability distribution
of a finite number of biotypes in a process of inheritance. The problem of a description of their trajectories
was stated by \citet{SUla1960}. Since the seventies of the $20^{\mathrm{th}}$ century limiting behaviour of
iterates of quadratic stochastic operators was intensively studied 
\citep[see e.g.][]{HKe1970,Ly1971,Va1972,Zak1978,NGanDZan2004,KBarMMis2010}.
The field is steadily evolving in many
directions \citep[see][for a detailed review of mathematical results and open problems]{RGanFMukURoz2011}.
Recently \citet{WBarMPul2013} introduced and examined in detail different
types of asymptotic behaviours of quadratic stochastic operators in the (discrete) $\ell^{1}$ case.
The results obtained there were subsequently generalized to the (continuous) $L^{1}$ case
by \citet{KBarMPul2015,KBarMPul2015BMMSS}.
Furthermore, \citet{KBarMPul2013} described an algorithm to simulate the behaviour of iterates of quadratic
stochastic operators acting on the $\ell^{1}$ space. However, it should be stressed that direct applications of quadratic
stochastic operators are still in their infancy even in a discrete case. Currently 
\citet{NGanNMoiWOthNHam2004}, \citet{NGanJDaoMUsm2010} and \citet{NGanMSabUJam2013} 
can serve as notable examples, which also
illustrate
the complexity of the concerned problem.
If one now starts to consider QSOs acting on $L^{1}$ then the situation becomes 
even more complicated, in a sense because Schur's lemma does not hold. 
To obtain results one needs to make restrictive assumptions 
on the QSO, e.g. \citet{KBarMPul2015} assume a kernel form (Definition \ref{dfKQSO}).
But even in this subclass it is not readily possible to prove convergence of a trajectory
of a QSO. Very recently \citet{RRudPZwo2015}
and \citet{PZwo2015} considered an even more restrictive 
subclass of kernel QSOs, 
that correspond to a model which
``retains the mean''
\citep[according to Eq. $(9)$ by][]{RRudPZwo2015}. With these
(and additional technical assumptions, like bounds on moment growth) they were able to obtain 
a convergence of the iterates, which is slightly stronger than the
weak convergence. 
Here, motivated by the model described in \citet{RRudPZwo2015}'s Example $1$
we consider a very special but biologically 
extremely relevant type of ``mean retention'' 
where the kernel of the QSO corresponds to an additive perturbation 
of the parents' traits 
(comp. Definition \ref{dfcQSOF}). This is of course less general than 
\citet{RRudPZwo2015}'s
Theorems $3$ and $4$, and substantially different than 
the particular case of their Eq. $(9)$.
First, we consider discrete time evolution. Moreover,
we are concerned with weak convergence only. 
But it is the price we pay for being allowed to drop the 
assumptions of kernel 
continuity, moment growth, technical bounds on 
elements of the death process and other elements of the 
continuous time 
process' generator and kernel. Instead, we need for the 
perturbing term 
either a finite second moment or control of the tails of its distribution by a power function, 
( Theorems \ref{thmWC} and \ref{thmWCpm}).
This does unfortunately also result in the loss
of uniqueness of the limit \citep[cf.][]{RRudPZwo2015,PZwo2015} --- it 
is seed specific. On the one hand this might
seem to a very serious drawback, certainly a global attractor is a more desirable 
result. But on the other hand there are numerous situations, e.g. computer simulations
of a system, where one is first interested if a system stabilizes when started from 
any initial condition,
not necessarily at the same state
(e.g. the classical Lotka--Volterra system does not
have a unique limiting cycle). Furthermore the conditions we provide
are very easy to verify, something which is desirable for an applied scientist.
More importantly the representation of the iterates of the QSO 
described by
Eq. (\ref{eqXYrep})
allows one to implement an efficient simulation 
algorithm and possibly
obtain convergence rates for a given QSO under study.

\section{Preliminaries}
Let $\left(X, \uA \right) $ be a separable measure space.
By $\M = \M\left(X, \uA , \| F \|_{TV} \right) $ we denote the Banach lattice of all 
signed measures on $X$ with finite total variation where the norm is given by

$$\| F \|_{TV} := \sup_{ f \in X } \{ | \langle F, f \rangle | : f \mbox{ is } 
\uA-\mbox{measurable, } \sup_{ x \in X } | f (x)| \le 1 \} , $$
where

$$  \langle F, f \rangle := \int_X \, f(x) \, dF(x).   $$
By $\uP := \uP\left(X, \uA\right)$ we denote the convex set of all probability measures on
$\left(X, \uA \right) $. 
Spaces constructed as the above $\M$, or appropriate subspaces of such  $\M$,
can serve as the state--spaces of processes describing the 
evolution of probability distributions
of $X$--valued traits of some species of interest. This evolution maybe
governed for example by the 
below concept of a quadratic stochastic operator, 
extending \citet{KBarMPul2015}'s definition.
Let  $\M_0$ be a Banach subspace  on $\M$, and  let  ${\uP}_0 := \uP \cap \M_0$.

\begin{defn}\label{dfQSOfmod}
A bilinear symmetric operator $\opQ \colon \M_0 \times \M_0 \to \M_0$ is called a
\emph{quadratic stochastic operator} on $\M_0$ if

$$ 
\|\opQ(F_1, F_2) \|_{TV} = \|F_1\|_{TV} \|F_2\|_{TV} \, \textrm { for all }  F_1, F_2 \in \M_0, \;
\textrm { and  }\opQ(F_1, F_2) \geq 0\ ,  \textrm{ if }F_1, F_2 \geq 0 \ .  
$$
\end{defn}
Notice that QSOs are bounded as 
$\sup_{\|F_1\|_{TV} = 1, \|F_2\|_{TV} = 1} \| \opQ
(F_1,F_2)\|_{TV} = 1$.                                   
Moreover,  if $\widetilde{F} \ge F \ge 0$ and $\widetilde{G} \ge G \ge 0$ then 
$\opQ(\widetilde{F},\widetilde{G}) \ge \opQ(F,G)$.
Clearly, $\opQ ( {\uP}_0 \times {\uP}_0) 
\  \subseteq \ {\uP}_{0}$.
Such QSOs have an interpretation in evolutionary biology.
Namely, imagine that we observe two populations, where  $F_{1}, F_{2} \in {\uP}_{0}$ 
represent their trait distributions.
Then $\mathbf{Q}(F_{1}, F_{2}) \in {\uP}_0$ represents a distribution of this trait 
in the next generation coming from 
 the mating of independent individuals from two different 
populations. Special attention is paid to the nonlinear ``diagonalized''
mapping
${\uP}_0 \ni F \mapsto \mathbb{Q}(F) := \mathbf{Q}(F,F) \in {\uP}_{0}$.
Then the values of the sequence of
iterates $\mathbb{Q}^{n}(F)$,  $n = 0,1,2,\ldots$,
model the evolution of the probability distribution of 
the $X$--valued trait of an inbreeding or hermaphroditic population, 
with  $F$  as the initial distribution.
Hence a typical question when working with quadratic stochastic operators is their 
long--term behaviour.

Let us reformulate the different types of asymptotic behaviour  
 of quadratic stochastic operators, originally
considered by \citet{WBarMPul2013,KBarMPul2015}.

\begin{defn}   
A quadratic stochastic operator $\mathbf{Q}$ on the Banach subspace  $\M_0$ is called:
\begin{enumerate}
\item \emph{norm mixing} (also called \emph{uniformly asymptotically stable}) 
if there exists a probability measure
$H\in {\uP}_{0} $ such that

$$\lim\limits_{n\to\infty}\sup_{F\in
{\uP}_0} \left\| \mathbb{Q}^{n}(F) - H \right\|_{TV} = 0 \ , $$
\item \emph{strong mixing} (\emph{asymptotically stable}) if there exists a probability measure
$H\in {\uP}_0$ such that for all $F \in {\uP}_0 $ we have

$$\lim\limits_{n\to\infty} \left\| \mathbb{Q}^{n}(F) - H
\right\|_{TV} = 0 \ , $$
\item \emph{strong almost mixing} if for all $F_1, F_2 \in {\uP}_0 $ we have

$$\lim\limits_{n\to\infty} \left\| \mathbb{Q}^{n}(F_{1}) - \mathbb{Q}^{n}(F_{2}) \right\|_{TV} = 0 \ .$$
\end{enumerate}
\end{defn}
\citet{KBarMPul2015} distinguished 
the kernel subclass of quadratic stochastic operators originally
defined on the Banach lattice 
$L^{1}(\mu) = L^{1}( X, \uA, \mu)$   of absolutely
integrable  real valued functions with respect to a fixed $\sigma$--finite 
positive  measure $\mu$. 
However $L^{1}(\mu)$ is  
isometrically isomorphic to the subspace 
$\M(\mu) := \M(X, \uA, \mu)$  of measures from $\M$ 
absolutely continuous with respect to $\mu$.
Therefore through the equality 
$L^{1} \ni f  \to 
(F(A) = \int_{A} f \ d\mu : A \in \uA)  \in 
\M(\mu)$, 
they may be equivalently defined as below. 

\begin{defn}\label{dfKQSO}
A  quadratic stochastic operator $\opQ \colon \M(\mu) \times
\M(\mu) \to \M(\mu)$  is called a \emph{kernel} quadratic stochastic
operator if there exists an $\uA \otimes \uA \otimes \uA$--measurable, nonnegative function $q
\colon X \times X \times X \to \mathbb{R}_{+}$, such that 
$q(x,y,z) = q(y,x,z)$ for any $x,y,z \in X$ ,
$\int_X q(x,y,z)d\mu (z) =1$ for $\mu\times \mu$-almost all 
$(x,y) \in X \times X$,
and $\opQ(F_{1},F_{2}) = \opQ_{q}(F_{1},F_{2})$, where

$$\opQ_{q}(F_1,F_2)(A) = \int_A\int_X \int_X f_1(x) f_2(y) q(x,y,z) \ud\mu(x) \ud\mu(y) \ud\mu(z), \quad  A \in \uA,$$
for measures $F_{i}$ with densities  $f_{i}$, $i=1,2$.
\end{defn}
\citet{KBarMPul2015} provide a detailed study of the 
limit behaviour of (kernel) quadratic stochastic operators.
In particular, equivalent conditions for (uniform) asymptotic stability of 
such operators are expressed  in terms of nonhomogeneous
chains of linear Markov operators.
There are of course other relevant works in the literature dealing with the topic of limit behaviour of quadratic stochastic operators. For instance,
\citet{NGanNHam2014,NGanMSabANaw2014,NGanRGanAJam2014}
recently studied non--ergodicity of QSOs.

Many models do not require the strong convergence of 
the considered trait distributions. Weak convergence, especially for vector valued 
traits or for
those concentrated on finite sets, seems perfectly sufficient. 
Therefore we introduce another type of long--term 
behaviour of quadratic stochastic operators based on the weak convergence of measures.  
With this in mind in what follows we make the below crucial assumption about $\M$.
\begin{assm}
$\M$ is the Banach lattice of finite Borel measures on 
the trait value space $\langle X, \uA \rangle$ equal to a complete separable 
metric space with $\uA$ consisting of Borel sets (generated by the open subsets of  $X$). 
\end{assm}
Then a sequence of measures $F_{n} \in \M$ is said to be 
\emph{weakly convergent} to a  
measure $ H \in \M$,  if for every continuous bounded function
$f \in C(X)$  the functionals  $\langle F_n, f\rangle $  approach  
$\langle H, f\rangle $, as  $n \to  \infty$ 
\citep[cf.][]{MBil1979}.   

\begin{defn}\label{weakasKSO}
The quadratic stochastic operator $\mathbf{Q}$ on $\M$ is said to be 
\emph{weakly asymptotically stable} at $F\in \uP$ if there is an $H\in \uP $, such that  
the sequence of probability measures 
${\mathbb{Q}^{n}(F)}_{n \in \mathbb{N}}$
converges weakly to $H$.
\end{defn}
In the next sections we study in details the situation where the trait values belong to 
finite dimensional real vector space. This natural setting allows us to exploit the 
apparatus of characteristic functions.

\section{The centred QSO in 
$\uR^{d}$}
We will focus on a very specific subclass of 
quadratic stochastic operators which we call \emph{centred}.
For this we assume $X= \uR^d$, $d \in \mathbb{N}_{+}$, for the trait value space.
Thus the state space equals the lattice $\M^{(d)}= \M(\uR^d, \mathcal{B}^{(d)})$ of all finite 
Borel measures on  $\uR^{d}$ with finite variation. 
The corresponding probability distributions form a convex subset denoted  
by $\uP^{(d)} = \uP(\uR^d, \mathcal{B}^{(d)})$. The convolution of elements of  
$\M^{(d)}$ is defined as

$$   F \ast  G\, (A) := \int_{\uR^d} \ F(A - y ) \ \ud G(y), \quad  A \in  \mathcal{B}^{(d)}. $$
For any $F \in \M^{(d)}$ and $n\in \uN_1$ we write
$\begin{array}{rc}
 & \\
F^{\ast n} := &
\underbrace{F \ast F \ast \cdots \ast F}\\
& \footnotesize{
 \;\;n \;\mbox{  factors }}
\end{array}
$
for the $n$--th convolutive power  of $F$. Furthermore, 
we write $\dot{F}(A) := F(2\cdot A)$, $A\in {\uR}^d$.
\begin{defn}\label{dfcQSOF}
Let $G \in \uP^{(d)}$. The associated with $G$ operator $\opQ_G \colon 
\M^{(d)} \times \M^{(d)} \to \M^{(d)}$ defined by
\be
\mathbf{Q}_{G}(F_{1},F_{2})  := \dot{F}_{1}\ast\dot{F}_{2}\ast G,
\ee
is called a \emph{centred quadratic stochastic operator}. 
If additionally $G$ is absolutely continuous with respect to the Lebesgue measure 
$\lambda^{(d)}$, then
$\mathbf{Q}_{G}$ is called a \emph{centred kernel  quadratic stochastic operator}.
\end{defn}
We omit the straightforward proof, that the above defined operator $\mathbf{Q}_{G}$ is a QSO. 
The following example briefly explains why for absolutely continuous 
measures $G$ the name \emph{kernel} introduced in Definition \ref{dfcQSOF} is applicable. 

\begin{exam}
If $F_{1}, F_{2}, G$ 
are probability measures on $\uR^d$ absolutely continuous with respect to
the Lebesgue measure $\mathcal{\lambda}^{(d)}$
then their densities 
$f_{1} := \tfrac{\ud F_{1}}{\ud {\lambda}^{(d)}}, f_{2} := \tfrac{\ud F_{2}}{\ud {\lambda}^{(d)}}, 
g := \tfrac{\ud G}{\ud {\lambda}^{(d)}}$
are elements of 
$L^{1} = L^{1}(\mathbb{R}^{(d)}, \mathcal{B}^{(d)}, \lambda^{(d)})$
and we may write

$$
\tfrac{\ud}{\ud \lambda^{(d)}}\opQ_{G}(F_{1},F_{2})(z) =
\int_{\mathbb{R}}\int_{\mathbb{R}} f_{1}(x)f_{2}(y) \, g(z-\tfrac{x+y}{2})\, \ud x \, \ud y, \; z \in \mathbb{R}.
$$
Indeed, denoting for any $f, h \in L^{1}$ their (density--type) convolution 
by 

$$f \circledast h (z) := \int_{\mathbb{R}^{d}} f(z-y) \, h(y) \ud y,$$
for $z \in \mathbb{R}^{d}$, we have

$$
\begin{aligned}
\int_{\mathbb{R}^{d}}\int_{\mathbb{R}^{d}} f_{1}(x)\, f_{2}(y)\, g(z-\tfrac{x+y}{2})\, \ud x \, \ud y  
&=
\int_{\mathbb{R}^{d}} f_{1}(x)\left(\int_{\mathbb{R}^{d}} \dot{f}_{2}(y)\, 
g((z-\tfrac{x}{2})-y)\, \ud y \right) \ud x
\\ &=
\int_{\mathbb{R}} f_{1}(x)\left( \dot{f}_{2} \circledast g(z-\tfrac{x}{2}) \right) \ud x
\\ &=
\dot{f}_{1} \circledast \dot{f}_{2} \circledast g (z) \\
&= \tfrac{\ud}{\ud {\lambda}^{(d)}}\opQ_{G}(F_{1},F_{2})(z),
			\end{aligned}
$$
since $\dot{f}_{i} (x) := {\ud \dot{F}_{i} \over \ud \lambda^{(d)}}(x) = 2 f_{i} 
(2 \cdot  x)$.  
Thus, 
according to Definition 
\ref{dfKQSO}, $\opQ_{G}$ equals  
the
kernel quadratic stochastic operator
$\mathbf{Q}_{q}$ on  $L^1(\lambda^{(d)})$ 
with $q(x,y,z) = g(z-\tfrac{x+y}{2})$, $x,y,z \in \mathbb{R}^d$.  
\end{exam}
As before, 
we pay special attention to the corresponding ``diagonalized''
mapping

\be\label{defbbQ}
\M^{(d)} \ni F \mapsto \mathbb{Q}_{G}(F):= \opQ_{G}(F,F) \in \M^{(d)}, \ee
where $G \in \uP^{(d)}$ is 
arbitrarily fixed.  
For a given  $F \in \M^{(d)}$ and  natural number $n$ we  denote the result of the $n$--th 
iterate by

\be\label{defbbQn}
H^{\text{\circledchar{$n$}}}:=(\mathbb{Q}_{G})^{n}(F). \ee 
For any $H\in  \M^{(d)}$ 
we define its characteristic function by

$$\varphi_H(s) :=  \int_{\uR^{d}} \ \exp\left(
i\,  s \cdot x 
\right) \ \ud H(x) , \;  s \in {\uR^{d}},$$
where  $\cdot$ stands for the canonical scalar product in ${\uR^{d}}$.
In these terms   we have
\be\label{eqphiQn}
\varphi_{H^{\text{\circledchar{$n$}}}}(s) = \left( \varphi_{F}\left(\tfrac{s}{2^{n}}\right)\right)^{2^{n}} \ \prod\limits_{j=0}^{n-1}\left(\varphi_{G}\left(\tfrac{s}{2^{j}}\right) \right)^{2^{j}}, \ s \in \mathbb{R}, \, n \in {\uN}_{+}.
\ee
Indeed, first notice that

$$
\varphi_{H^{\text{\circledchar{$1$}}}}(s) = \varphi_{\mathbb{Q}_{G}(F)}(s)
 =\varphi_{\dot{F}}(s) \cdot \varphi_{\dot{F}}(s) \cdot \varphi_{G}(s) =
\left( \varphi_{F}\left(\tfrac{s}{2}\right) \right)^{2} \cdot \varphi_{G}(s) .
$$
Similarly, for $m := n+1$ from the $n$--th equation we get

$$
\begin{aligned}
\varphi_{H^{\text{\circledchar{$m$}}}}(s) &= 
\varphi_{\mathbb{Q}_{G}(H^{\text{\circledchar{$n$}}})}(s)
= \left(\varphi_{H^{\text{\circledchar{$n$}}}}
\left(\tfrac{s}{2}\right) \right)^{2} \cdot \varphi_{G}(s) \\
&= \ldots = \left( \varphi_{F}\left(\tfrac{s/2}{2^{n}}\right) \right)^{2^{n} \cdot 2} \cdot \prod\limits_{j=0}^{n-1} \left( \varphi_{G}\left(\tfrac{s/2}{2^{j}}\right)\right)^{2^{j}\cdot 2} \cdot \varphi_{G}(s) \\
&= \left(\varphi_{F}\left(\tfrac{s}{2^{n+1}}\right) \right)^{2^{n+1}} \cdot \prod\limits_{j=0}^{n} \left( \varphi_{G}\left(\tfrac{s}{2^{j}}\right)\right)^{2^{j}} .
\end{aligned}
$$
Thus, by induction Eq. (\ref{eqphiQn}) holds for all natural $n$. 
Consequently, by the L{\'e}vy-Cram{\'e}r continuity theorem
\citep[Theorem 3.1, Chapter 13][]{ShG00PrS} and
paraphrasing Definition \ref{weakasKSO} 
we may say that 
for $G \in \uP^{(d)}$ a centred 
quadratic stochastic operator $\mathbf{Q}_{G}$ is  weakly
 asymptotically stable at  $F\in \uP^{(d)}$ if for some $H  
\in \uP^{(d)}$ and  for every  $s \in \uR^d$
the characteristic function $\varphi_{H^{\text{\circledchar{$m$}}}}(s)$ approaches  $\varphi_{H}(s)$
as $n\to \infty$.

As we  show below,  the dependence on  the  distribution of $F$  
is substantial. However equality of the limit for different  
initial distributions
can be achieved when $F$ belongs to suitable subclasses. 
An exemplary subclass consists of
distributions with common finite mean value 
(whenever the limit exists for at least one of its members).

\section{Main results}\label{secMainRes}

Let us fix $F,G \in \uP^{(d)}$. 
The QSO  $ \mathbb{Q}_G$, acting on $(\uR^{d}, 
\mathcal{B}^{(d)})$,  
 is defined by Eq. (\ref{defbbQ}),  
and the values of its iterates are given by Eq. (\ref{defbbQn}). 
Moreover, for any  probability distribution  $H\in \uP^{(d)}$,
$\varphi_{H}$  denotes its characteristic function, and the vector of moments of order $1$ 
and the covariance matrix are defined by

$$   
\begin{array}{rcl}
m^{(1)}_{H} & := &
\displaystyle   
 \int_{\uR^{d}} \ x \ \ud H(x) = 
\left[\int_{\uR^{d}} \ x_{j} \  
 \ud H(x) \ \colon \
j \in \{1,2,\dots,d\}\right], \quad
\\
v_{H} & := &
\displaystyle   
\left[ \int_{\uR^d} \ x_{j} x_{k} 
\,
 \ud H(x) \ \colon  \  (j,k) \in 
\{1,2,\dots,d\}^{2} \right],
\end{array}
$$
whenever they exist in  $\uR^{d}$  and  $\uR^{d\times d}$, respectively.

\begin{thm}\label{thmWC}
Let $F\in \uP^{(d)}$ have finite first moments  $m := m^{(1)}_F \in
\uR^{d}$ and let all first and second moments of $G \in \uP^{(d)}$ be finite.
We set 
$m^{(1)}_G = 0\in \uR^{d}$, 
and $v := v_{G} \in \uR^{d\times d}$.
Then	$\uQ_{G}$ is weakly stable at $F$ or more precisely,
 the sequence $(H^{\text{\circledchar{$n$}}})_{n \in \mathbb{N}} \subset \uP^{(d)}$
converges weakly to $H^{\text{\circledchar{$\infty$}}} \in \uP^{(d)}$ 
with characteristic function equal to

\be\label{eqndefGniesk}
\varphi_{H^{\text{\circledchar{$\infty$}}}}(s) = e^{im\cdot s}
\varphi_{G^{\text{\circledchar{$\infty$}}}}(s), 
\textrm{  where  }
\varphi_{G^{\text{\circledchar{$\infty$}}}}(s) := 
 \lim\limits_{n\to \infty} \ \varphi_{G^{\text{\circledchar{$n$}}}}(s), \ s \in \mathbb{R}^d,
\ee
\be\label{eqndefGsk}
\varphi_{G^{\text{\circledchar{$n$}}}}(s) := 
 \prod\limits_{j=0}^{n-1} 
\left(\varphi_{G}\left(\tfrac{s}{2^{j}}\right) \right)^{2^{j}}, \ s \in \mathbb{R}^d.
\ee
\end{thm}

\begin{proof}
According to Eq. (\ref{eqphiQn})
 for any natural number $n$,  
$\varphi_{H^{\text{\circledchar{$n$}}}}(s)$
is a characteristic function of      
the random $d$--dimensional vector 
$Z^{\text{\circledchar{$n$}}}:=
X^{\text{\circledchar{$n$}}}+Y^{\text{\circledchar{$n$}}}$, where

\be\label{eqXYrep}
\begin{aligned}
X^{\text{\circledchar{$n$}}} &:= \frac{X_{1}+X_{2}+\ldots+X_{2^{n}}}{2^{n}}, \\
Y^{\text{\circledchar{$n$}}} &:= \sum\limits_{j=0}^{n-1}
\frac{Y_{1}^{(j)}+Y_{2}^{(j)}+\ldots+Y_{2^{j}}^{(j)}}{2^{j}}
\end{aligned}
\ee
and $X_{1},X_{2},X_{3},\ldots$ and
$Y_{1}^{(0)},Y_{1}^{(1)},Y_{2}^{(1)},Y_{1}^{(2)}, \ldots, Y_{4}^{(2)},\ldots, Y_{1}^{(j)}, \ldots, Y_{2^{j}}^{(j)},\ldots$
are independent sequences of random vectors such that
$X_{1},X_{2},X_{3},\ldots$ are independent identically distributed 
according to $F$
and $Y_{1}^{(0)},Y_{1}^{(1)},Y_{2}^{(1)},Y_{1}^{(2)}, \ldots, Y_{4}^{(2)},\ldots, Y_{1}^{(j)}, \ldots, Y_{2^{j}}^{(j)},\ldots$
are independent identically distributed 
according $G$.
Since $m^{(1)}_{F}=: m \in \mathbb{R}^d$, by 
the Strong Law of Large Numbers we obtain that
$\lim\limits_{n \to \infty} X^{\text{\circledchar{$n$}}} = m$ almost surely. 
Hence for the first factor of  (\ref{eqphiQn}) we have

$$ \lim\limits_{n \to \infty} \left( \varphi_{F}\left(\tfrac{s}{2^{n}}\right)\right)^{2^{n}} = e^{im\cdot s}.$$
The assumptions taken on $G$ imply that the 
covariance matrix of the independent random vectors 

$U_{j} := (Y_{1}^{(j)}+Y_{2}^{(j)}+\ldots+Y_{2^{j}}^{(j)})/2^{j}$, $j = 0,1,2 \ldots$, 
equals $v/2^{j}$ for any $j = 0,1,2,\ldots$,
and hence the series
$Y^{\text{\circledchar{$\infty$}}} := \sum_{j=0}^{\infty}U_{j}$
converges almost surely \citep[as it converges coordinatewise, cf.][Theorem 2.5.3]{RDur2010}.
Thus, the probability distribution  $G^{\text{\circledchar{$n$}}}$  
of $Y^{\text{\circledchar{$n$}}}$
converges weakly to the probability distribution
of $Y^{\text{\circledchar{$\infty$}}}$. Therefore, 
again by the continuity theorem,  the limit $\varphi_{G^{\text{\circledchar{$\infty$}}}}(s)$ of the 
second factor of  (\ref{eqphiQn}) is the characteristic function 
of the  probability  distribution  
of $Y^{\text{\circledchar{$\infty$}}}$.
By all of the above we obtain that
$\varphi_{H^{\text{\circledchar{$n$}}}}(s) 
\to e^{im\cdot s} 
\varphi_{G^{\text{\circledchar{$\infty$}}}}(s) $
for every $s \in \mathbb{R}$, where the limiting function is 
the characteristic function  of 
the probability distribution of
$m + Y^{\text{\circledchar{$\infty$}}}$.
\end{proof}
\begin{rem}
We can observe that after a large number of iterations the starting distribution
is only responsible for the expectation of the law of $\mathbb{Q}_{H^{\text{\circledchar{$n$}}}}(\cdot)$.
It would be tempting to suspect a central limit theorem will hold for kernel
part described by $Y^{\text{\circledchar{$n$}}}$. However this will not occur as
$U_{j}$ tends almost surely to $0$. Or
in other words the tail elements of the product defining
$\varphi_{G^{\text{\circledchar{$\infty$}}}}(s)$ 
tend to $1$. This means that the limiting distribution of
$Y^{\text{\circledchar{$n$}}}$ essentially depends on the initial
elements of the sequence $\{U_{j}\}$ and not on the tail ones
``normalizing everything'' as is the case for CLTs. 
This makes it difficult to make closed form statements about
the law of $G^{\text{\circledchar{$\infty$}}}$.
\end{rem}

For further analysis we confine ourselves to the one dimensional
 case $(d = 1).$ Due to the factorization of the characteristic function of 
$\varphi_{H^{\text{\circledchar{$n$}}}}$ expressed by Eq. (\ref{eqphiQn}), the
sufficient conditions for stability are divided into two steps -- first the 
existence of the limit probability distributions of the arithmetic means 
$X^{\text{\circledchar{$n$}}}$ and  second --- of the 
weighted sums $Y^{\text{\circledchar{$n$}}}$, both  given by Eq. (\ref{eqXYrep}).
For the former, we apply the well known theory of stable probability 
distributions 
\citep[see e.g. Chapter 17 and Defn. 2.2 in Chapter 15 of][respectively]{FeW66IPT,ShG00PrS}. 
Accordingly we can state the following.

\begin{thm}\label{thmWCpm}  

~
\vspace{1ex}

\noindent (i) 
If the weak limit $F^{\text{\circledchar{$\infty$}}}$ of the probability
distributions    
$F^{\text{\circledchar{$n$}}}$ of $X^{\text{\circledchar{$n$}}}$  exists 
 in $\uP^{(1)}$,  then the limiting characteristic function $\varphi_{F^{\text{\circledchar{$\infty$}}}}$ 
satisfies the equation 
 
$$\varphi_{F^{\text{\circledchar{$\infty$}}}}(2s ) =   \left(\varphi_{F^{\text{\circledchar{$\infty$}}}}(s ) \right)^{2},  \;  
\mbox{\rm for all} \;  s \in \uR.$$

\vspace{1ex}  

\noindent
(ii) 
If  the one-dimensional probability measure  $F\in \uP^{(1)}$ 
belongs to the strict domain of attraction of a  stable 
probability measure  with characteristic exponent 1,  
i.e. if  for some $S\in  \uP^{(1)}$  we have
$ (\varphi_{F^{\ast n}}({s\over n}))^n \to \varphi_{S}(s)$, 
as $n \to \infty$,  for  $ s \in \uR$,
then the weak limit distribution
${F^{\text{\circledchar{$\infty$}}}}$
of the averages $X^{\text{\circledchar{$n$}}}$ 
equals $S$, too, which is a Cauchy probability distribution, i.e. for some  
$c \ge 0$, $m \in \uR$

$$  \varphi_{F^{\text{\circledchar{$\infty$}}}}(s )  =  \exp\{ - c \vert s \vert + i m s \}, \quad  s \in  \uR.$$

\vspace{1ex}

\noindent
(iii)  
Let the one--dimensional probability distribution $ G \in \uP^{(1)}$ 
satisfy the following condition
\begin{itemize}
\item[] \ $\ln \varphi_{G}(s) \vert \le A \vert s \vert ^{p}$
for any 
$\vert s \vert \le s_{0} $
for some reals $ s_{0}> 0$, $p>1$, $A>0$.  
\end{itemize}
Then $(G^{\text{\circledchar{$n$}}})_{n \in \mathbb{N}}$
converges weakly to a probability distribution $G^{\text{\circledchar{$\infty$}}} \in \uP^{(1)}$
whose characteristic function is given by the infinite product of Eq. (\ref{eqndefGniesk}).

\vspace{1ex}

\noindent
(iv)
Under  the   conditions  of (i)   and (iii) 
$\left({H^{\text{\circledchar{$n$}}}}\right)_{n \in \uN} $ converges weakly to 
$ F^{\text{\circledchar{$\infty$}}} \ast G^{\text{\circledchar{$\infty$}}}$.  
\end{thm}

\begin{proof}
For the first claim, let  $n \to \infty$. Then  

$$\varphi_{F^{\text{\circledchar{$\infty$}}}}(2s) = 
\lim\limits_{n \to \infty} \varphi_{F}\left(\frac{2 s}{2^{n}}\right)^{2^{n}} = 
\lim\limits_{n \to \infty} \left(\varphi_{F}\left({\frac{s}{2^{n-1}}}\right)^{2^{n-1}}\right)^{2} =
\left(\varphi_{F^{\text{\circledchar{$\infty$}}}}(s)\right)^{2} .$$ 
By the assumptions of  (ii), equality $\varphi_{F^{\text{\circledchar{$\infty$}}}} = S$ is obvious. Moreover,
the equalities follow:  $ \varphi_S( n s ) = (\varphi_S( s ))^n$, for 
all $n \in \uN$,  $s \in \uR$. Thus the claim is a well known result on stable distributions with characteristic exponent 1, \citep[cf. Theorem 3.1 in][Chapter 15]{ShG00PrS} 

Due to the bounds on $\varphi_{G}$ we have that for any positive  real number $T$ 
there exists a natural number $J$ such that
for every $j \ge J$ and every $\vert s \vert < T$

$$
2^{j} \left\vert \ln \varphi_{G}(\tfrac{s}{2^{j}}) \right\vert \le A \tfrac{\vert s \vert ^{p}}{2^{j(p-1)}}.
$$
Therefore 
$\varphi_{G^{\text{\circledchar{$\infty$}}}}$ exists and     
is a characteristic function of a 
probability measure on
$\mathbb{R}$ as an almost uniform limit of
characteristic functions of probability measures,  proving (iii).     
The claim of  (iv) is a simple corollary to the defining formula, Eq. (\ref{eqphiQn}), for  
$\varphi_{H^{\text{\circledchar{$n$}}}}$.
\end{proof}

\begin{rem}\label {lemln}
Since  $\ln \varphi_G(0) =1$, the assumption on $\ln \varphi_{G}$ 
near zero can be equivalently replaced by the same estimates for $\varphi_{G} -1$. 
This  is a direct consequence of the following inequalities valid for all complex numbers $a$ with  $\vert a \vert  < 0.5$, 

$$    \vert  a \vert (1 - \vert  a \vert ) \le       \vert \ln ( 1+ a) \vert \le  \vert  a \vert (1 + \vert  a \vert ).$$
\end{rem}

\begin{thm}\label{necasuff} 

~\vspace{1ex}
 
\noindent
(i) 
Let $F, G \in \uP^{(1)}$. If a centred  
quadratic stochastic operator   $\mathbf{Q}_{G}$ 
is weakly asymptotically stable at  $F$ with the limit distribution 
$H^{\text{\circledchar{$\infty$}}}\in \uP$,
then the following equality is satisfied
\be\label{eqGprop}
\varphi_{H^{\text{\circledchar{$\infty$}}}}(s) = \left( \varphi_{H^{\text{\circledchar{$\infty$}}}}(\tfrac{s}{2}) \right)^{2} \varphi_{G}(s), \ s \in \mathbb{R}.
\ee

\vspace{1ex}

\noindent (ii)
The probability distribution $H \in \uP^{(1)}$ is a weak limit of the sequence
$\left(\mathbb{Q}_G^n(F)\right)_{n \in \mathbb{N}}$ for some  $F, G\in \uP^{(1)}$
if and only if for some characteristic function $\varphi$ of a 
probability measure on  $\uR$ the following equation holds

$$ \varphi_H(s) = \left( \varphi_H(\tfrac{s}{2}) \right)^{2} \varphi(s), \ s \in \mathbb{R}.  $$
Then, for $G$ one can  assume  $\varphi_{G} = \varphi$. 
Moreover, $H$ is a fixed point of  such  $\mathbb{Q}_{G}$, i.e.
 $\mathbb{Q}_{G}(H) = H$.
\end{thm}

\begin{proof}
From the definition of weak stability we have

$$
\varphi_{H^{\text{\circledchar{$\infty$}}}}(s) = \lim\limits_{n \to \infty} 
\left( \varphi_F(\tfrac{s}{2^n}) \right)^{2^n} \ \varphi_{G^{\text{\circledchar{$n$}}}}(s)
$$
for any real $s$, where  $\varphi_{G^{\text{\circledchar{$n$}}}}$ is defined through Eq. (\ref{eqndefGsk}). 
Clearly we then have

$$
\begin{array}{rl}  
\varphi_{G}(s) \left(\varphi_{H^{\text{\circledchar{$\infty$}}}}(\tfrac{s}{2})
\right)^{2} &=
\lim\limits_{n \to \infty} \left( \varphi_F(\tfrac{s}{2 \cdot 2^{n}}) \right)^{2 \cdot 2^{n}} 
\varphi_{G}(s) \prod\limits_{j=0}^{n-1} \left(\varphi_{G}(\tfrac{s}{2 \cdot 2^{j}})\right) 
\\ & \stackrel{m=n+1}{=}  
\lim\limits_{n \to \infty} \left( \varphi_F(\tfrac{s}{2^{m}}) 
\right)^{2^{m}} \
\varphi_{G^{\text{\circledchar{$m$}}}}(s) = \varphi_{H^{\text{\circledchar{$\infty$}}}}(s).
\end{array}
$$
Hence the first part is proved, as well as Eq. (\ref{eqGprop}),
which is equivalent to the following

$$  \mathbb{Q}_G\left(
{H^{\text{\circledchar{$\infty$}}}} \right) =
{H^{\text{\circledchar{$\infty$}}}}. $$
The completion of the second part follows directly from the assumption on  $\varphi$.
\end{proof}

\begin{rem}
We notice that  Theorem \ref{thmWCpm} is fulfilled
by the distributions  $F, G \in \uP^{(1)}$ of Theorem 
\ref{thmWC},
however, $\epsilon =1$ is not 	
covered by Theorem \ref{thmWCpm}.   
  Indeed, by the implied  Strong Law of Large Numbers, $F$ is in the strict domain of 
attraction to the stable probability distribution concentrated at the mean value $m^{(1)}(F)$.
This distribution   is    
stable with any exponent, in particular with exponent $1$. 
Moreover, the zero value of the mean
$m^{(1)}(G) = 0 $ jointly with the finiteness of the variance 
$v_G \left(= m^{(2)}(G)\right)$ implies that  $\varphi_G(s) = 1+ v s^2 + o(s^2)$, as $ s \to 0$.
Thus the second assumption of Theorem \ref{thmWCpm} is fulfilled with exponent
$p =2$. Obviously, for the particular case of $v=0$, i.e. of  $G $ 
concentrated at $0$, the claims of both theorems follow trivially, whenever 
$F$ is appropriate. Note, that this corresponds to the lack of any 
perturbation of  the arithmetic mean of the inherited trait. An interesting example of $F$ 
(due to P. L\'evy) where, 
$\varphi_{F}(n s ) =   \varphi^{n}_{F}(s )$  for only $n = 2^{k}$, $k=1,2, \ldots$,
can be found in \citet{FeW66IPT}'s Chapter 17, Section 3.
\end{rem}

\section{Some specific examples}

The condition on the logarithm of the kernel's characteristic function in Theorem \ref{thmWCpm}  
is not an ``uncommon'' one. Besides 
distributions $G$ with finite variance, there is a large class of heavy--tailed probability 
distributions (on $\uR$).
Specific subfamilies of this subclass are considered 
in the propositions below.       

\begin{prop}\label{PROP1}
If  $G\in \uP^{(1)}$ is symmetric and its tails for some constants  $C>0$ and 
$ \epsilon\in (0,1) $ satisfy

$$ G(-\infty,-x]= G[x, \infty)  \le C x^{-( 1 + \epsilon)}, ~~\mbox{for all } x >0,$$
then, possibly with another constant  $C >0 $, 
for every  $s \in \uR$ we have

$$\vert \ln \varphi_{G}(s) \vert < C\vert s \vert^{1+\epsilon},$$
or, equivalently, that 
$\vert 1- \varphi_{G}(s) \vert < C\vert s \vert^{1+\epsilon}.$      
In particular, the mean value of  $G$ exists (and equals 0).
\end{prop}

\begin{proof}
Due to the symmetry of $G$ it suffices to consider  positive  $s>0$.
Moreover,  for every  $A >0$  we have 

$$
\begin{array}{rl}
\displaystyle   
\vert 1-\varphi_{G}(s) \vert &= \displaystyle  \left \vert 1- \int\limits_{\uR} e^{isx} \ \ud G(x) \right \vert =
\displaystyle   
\left \vert \int\limits_{\uR} (1-
\cos(sx) 
) \ \ud G(x) \right \vert 
\\ &
\le 
\displaystyle   
2 \int\limits_{[0,A)} \frac{s^{2}x^{2}}{2} \ \ud G(x) +4 \int\limits_{[A, \infty)} \ \ud G(x).
\end{array}
$$
By  integration of the first term by parts, for  $A = \pi/s$  we obtain     

$$
\begin{array}{rl}
\displaystyle   
I & :=2 \int\limits_{[0,A)} \frac{s^{2}x^{2}}{2} \ \ud G(x)
=
\displaystyle   
- 2\frac{s^2 A^2}{2} ( 1 - G(A) )  + 2 s^2 \int\limits_{(0,A)} x G([x, \infty)) \ud x\\
&
\displaystyle   
\le 2 s^{2} \int\limits_{[0,A)} \ C \vert x\vert^{-\epsilon} \  \ud x = 2 C \pi^{1-\epsilon} s^{1+\epsilon}.
\end{array}
$$
For the second term, again for  $A = \pi/s $ we have

$$
II  := 4 \int_{[A, \infty)} dG(x) = 4 G([A, \infty) \le  4C \pi^{-(1+\epsilon)} s^{1+\epsilon}. 
$$
\end{proof}
Let us note, that Markov's inequality implies the assumed estimates of tails from finiteness of the 
absolute moment of order $p=1 + \epsilon$,
since for positive $x >0$ we have

$$ G(-\infty, -x]) + G([x, \infty))  \le \int_{\uR} \ \vert u \vert^{p} \ud G(u) / x^{p}. $$   
Next, let us point at the a stronger claim of the asymptotic behaviour of the characteristic function 
near the origin based on the theory of stability, obviously under stronger assumptions. For practical 
modelling of real events we present the following  corollary to the
second part in \citet{BoA72KTV}'s proof of his Theorem $5$
(Chapter $7$, Section $4$).

\begin{cor}\label{stabBor}
If  $G\in \uP^{(1)}$ with mean value $0$ possesses tails such that for some constants  $C>0$ and $ \epsilon\in (0,1)$
when $ x \to \infty$ behave as

\be \label{WARSTAB} G(-\infty,-x]= C x^{-( 1 + \epsilon)} + o\left(x^{-( 1 + \epsilon)}\right), \;
 G[x, \infty)  = C x^{-( 1 + \epsilon)}
+ o\left(x^{-( 1 + \epsilon)}\right)
\ee
then 

$$
\varphi_G(s)  - 1 = -
2 C c(\epsilon) \vert s\vert^{1 + \epsilon}  + o(\vert s\vert^{1 + \epsilon}), \; \mbox{ as }\; s \to 0,
$$
where   $c(\epsilon) = \frac{(1+\epsilon) \Gamma(1-\epsilon) \sin( \epsilon {\pi \over 2})}{\epsilon}$.
\end{cor}
Similar claims will hold for negative $\epsilon$, but we are not concerned with this situation.
However, the case of $\epsilon =0$ is important due to the fact, that functions with such tails are in 
the domain of attraction to stable probability distributions with characteristic exponent equal to $1$ 
and not concentrated at a single point. An easily checked 
set of sufficient conditions is given as follows 
\citep[see again Theorem 5, Chapter 7, Section 4 of][]{BoA72KTV,FeW66IPT}. 

\begin{cor}\label{stabBor1}
Under the assumption of Eq. (\ref{WARSTAB}) with $\epsilon = 0$,
 the characteristic 
function  $\varphi_{G}$, for a symmetric $G$, satisfies the following limit behaviour

$$  \lim_{n\to \infty}  \  \left( \varphi_{G} ( \tfrac{s}{n} ) \right)^n = 
\exp( - C \pi \vert s \vert ), \; \mbox{ for all }\;  s \in \uR.$$
\end{cor}

\begin{cor}\label{examples1}
Every Cauchy--like probability distribution with density

$$  \frac{\ud F}{\ud \lambda}(x) =  
C \left( 1  +  a \vert x - \mu \vert^\alpha \right)^{2 \over \alpha}, \; x \in \uR, $$
is in the domain of attraction to the Cauchy probability distribution. In particular it satisfies the conditions of
part (i) of Theorem \ref{thmWCpm}.
\end{cor}

\begin{cor}\label{examples2}
Every symmetric probability distribution $G$ satisfying one of  the below listed requirements fulfills also 
the assumptions of part (ii) of Theorem \ref{thmWCpm}:
\begin{enumerate}
\item $G(-\infty, x] =  \frac{ ( u(x) )^{\alpha} }{ (v(x))^{\beta} },  \; x>0$,
is a positive   increasing function not exceeding ${1\over 2}$, where
 $u$ and  $v$ are polynomials of degree $l$ and  $m$, respectively,
 with $ \epsilon := m \cdot \beta -  l \cdot \alpha -1 \in (0, 1)$; 
\item $G\{ \uZ\} =1$  and  $F\{j\}  = 
\frac{ ( u(j) )^{\alpha} }{ (v(j))^{\beta} },  \; j>0$, where 
$u$ and  $v$ are positive on $\uZ_+$ polynomials of degrees $l$ and  $m$ respectively,
with $ \epsilon := m \cdot \beta -  l \cdot \alpha -2 \in (0, 1)$;
\item $G$ is a stable probability distribution of characteristic exponent $1+ \epsilon$, i.e.  
with characteristic function  $\varphi_G(s) =  \exp( - \vert s\vert^{1+ \epsilon})$, $ s \in \uR$.
\end{enumerate}
\end{cor}
Thus, in particular, a discrete random variable  $Y$ with values in
$\mathbb{Z}\setminus\{0\}$ with probabilities

$$
P(X=k) = C \frac{1}{\vert k \vert^{2 + \epsilon}}~~\mbox{for}~~k\neq 0,~~\mbox{where}~~\epsilon>0.
$$
can serve as a model of perturbation of the inherited trait.
The generated QSO 
$\mathbb{Q}_{G}$ is stable at every $F$ satisfying the requirements of Theorem \ref{thmWC}.  

\begin{rem}
Notice that nowhere in this work do we require that the limit of the operator is
unique, only that it is to exist for a seed distribution $F$ 
satisfying the conditions of Theorems \ref{thmWC} or \ref{thmWCpm}.
We also present weak convergence results and we suspect that
in many practical cases it will not be possible to obtain strong convergence ($L^{1}$)
results. It is very plausible that by the law of large numbers we will observe
convergence to a Dirac $\delta$. This is an important situation as it indicates fixation
of a population and weak convergence handles it perfectly well. However strong
convergence will not detect this stabilization. All the iterates of the operator
can produce smooth densities hence we will observe an $L^{1}$ distance of $2$ between all of
the iterates and the final measure.
\end{rem}

\section{Simulation algorithm}
\citet{KBarMPul2013} discussed how simulating quadratic stochastic operators acting on $\ell^{1}$ differs
from simulating the trajectory of a Markov linear operator. In the $L^{1}$ case we can employ the same
procedure to simulate a population behaving according to a kernel quadratic stochastic operator
acting on $L^{1} \times L^{1}$. We describe it in Algorithm \ref{algQg} and
Fig. \ref{figSimDensAlgQg} presents histograms from an example run.
For the simulations presented in 
Fig. \ref{figSimDensAlgQg} we considered the normal density kernel with variance $0.5$, i.e.
\be\label{eqSimQ}
q(x,y,z)  = \frac{1}{\sqrt{\pi}} \exp(-(z-(x+y)/2)^{2}) \equiv g(x-\frac{x+y}{2}).
\ee
One can directly verify that any normal distribution with unit variance, $\mathcal{N}(\mu,1)$, 
is a fixed point of the QSO.
\begin{algorithm}[!ht]
\caption{Simulating
$\mathbb{Q}(g)$}\label{algQg}
\begin{algorithmic}
\STATE Draw $K$ independent individuals according to the law of $g$ and call them $P_{0}$
\FOR{$i=1$ to $n$}
\STATE $P_{i}:=\emptyset$
\FOR{$j=1$ to $K$}
\STATE Draw a pair $(x_{j},y_{j})$ of individuals from population $P_{i-1}$
\STATE Draw an individual $z_{j}$ according to the law of $\mathbf{Q}(\delta{x_{j}},\delta{y_{j}})=q(x_{j},y_{j},\cdot)$
\STATE $P_{i}=P_{i} \cup \{z_{j} \}$
\ENDFOR
\ENDFOR
\STATE \textbf{return} $P_{0},P_{1},\ldots,P_{n}$
\end{algorithmic}
\end{algorithm}
\begin{figure}[!ht]
\begin{center}
\includegraphics[width=0.32\textwidth]{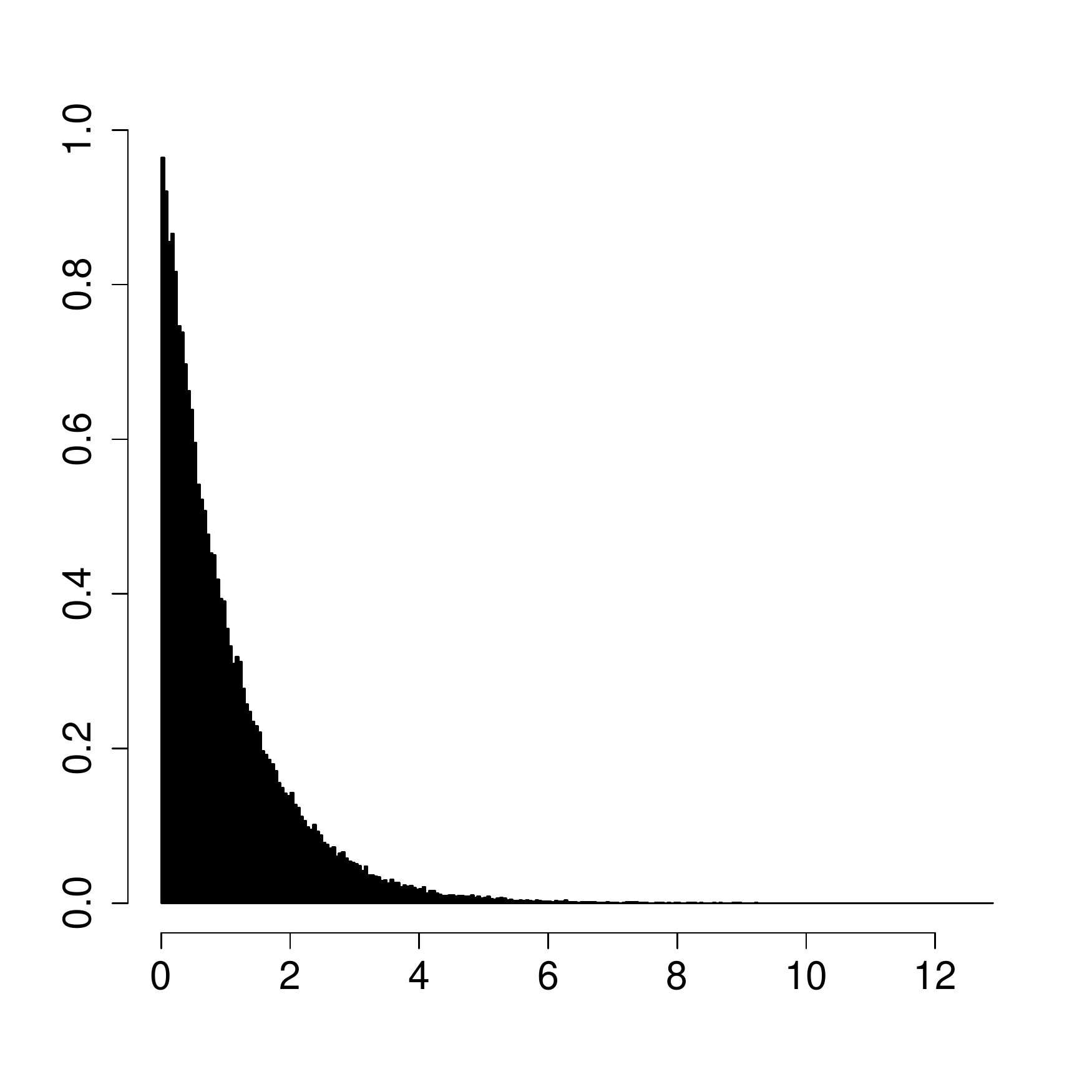}
\includegraphics[width=0.32\textwidth]{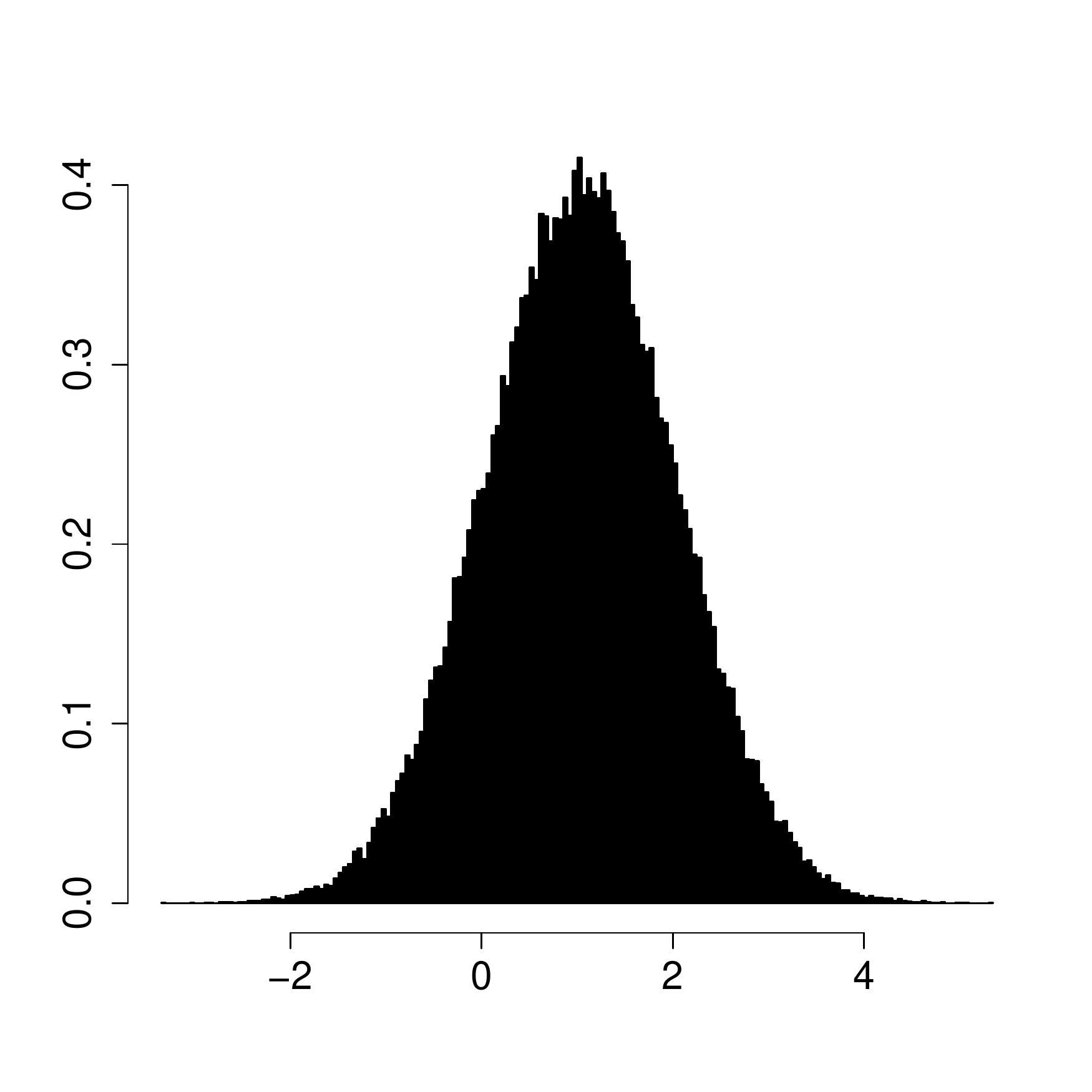}
\includegraphics[width=0.32\textwidth]{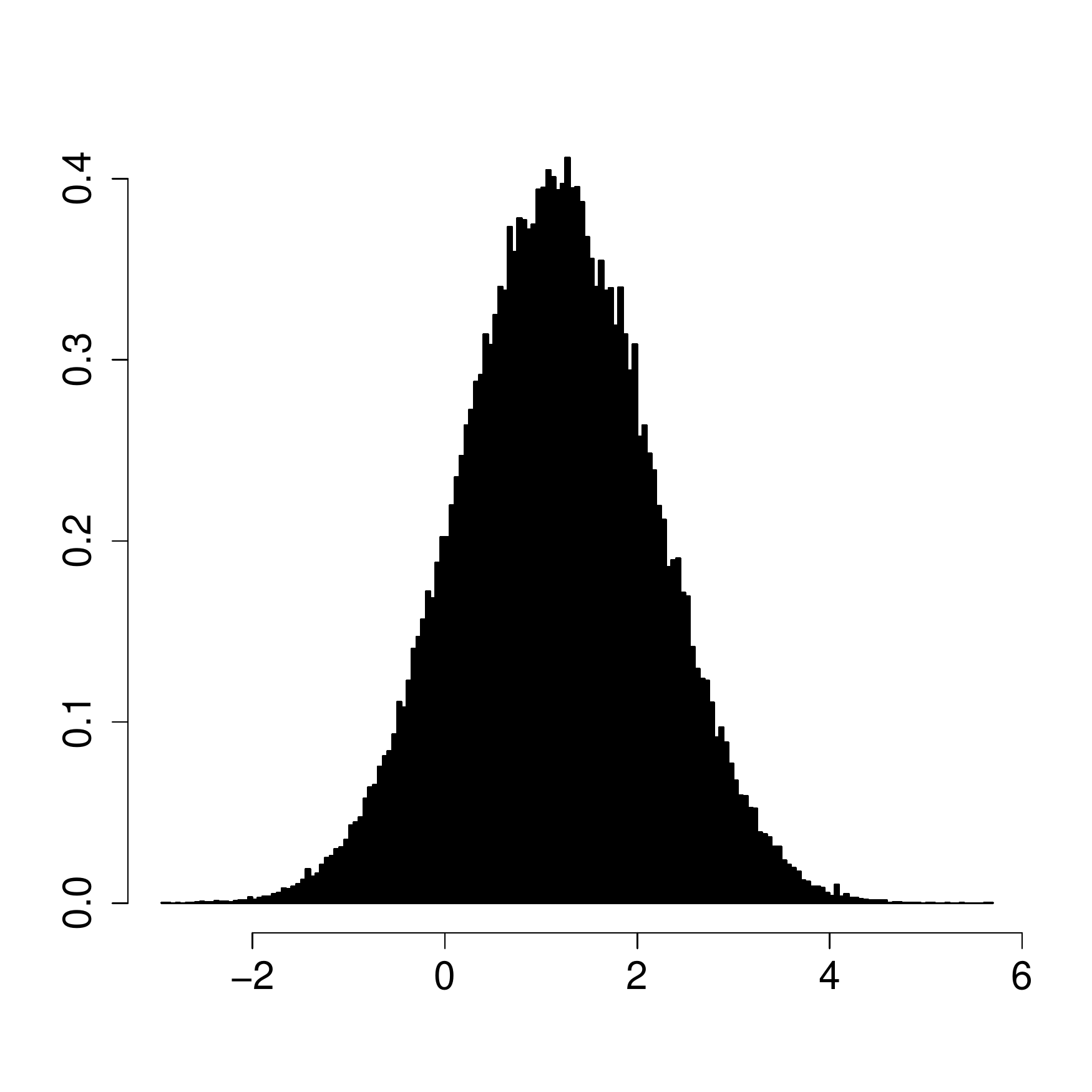} \\
\includegraphics[width=0.32\textwidth]{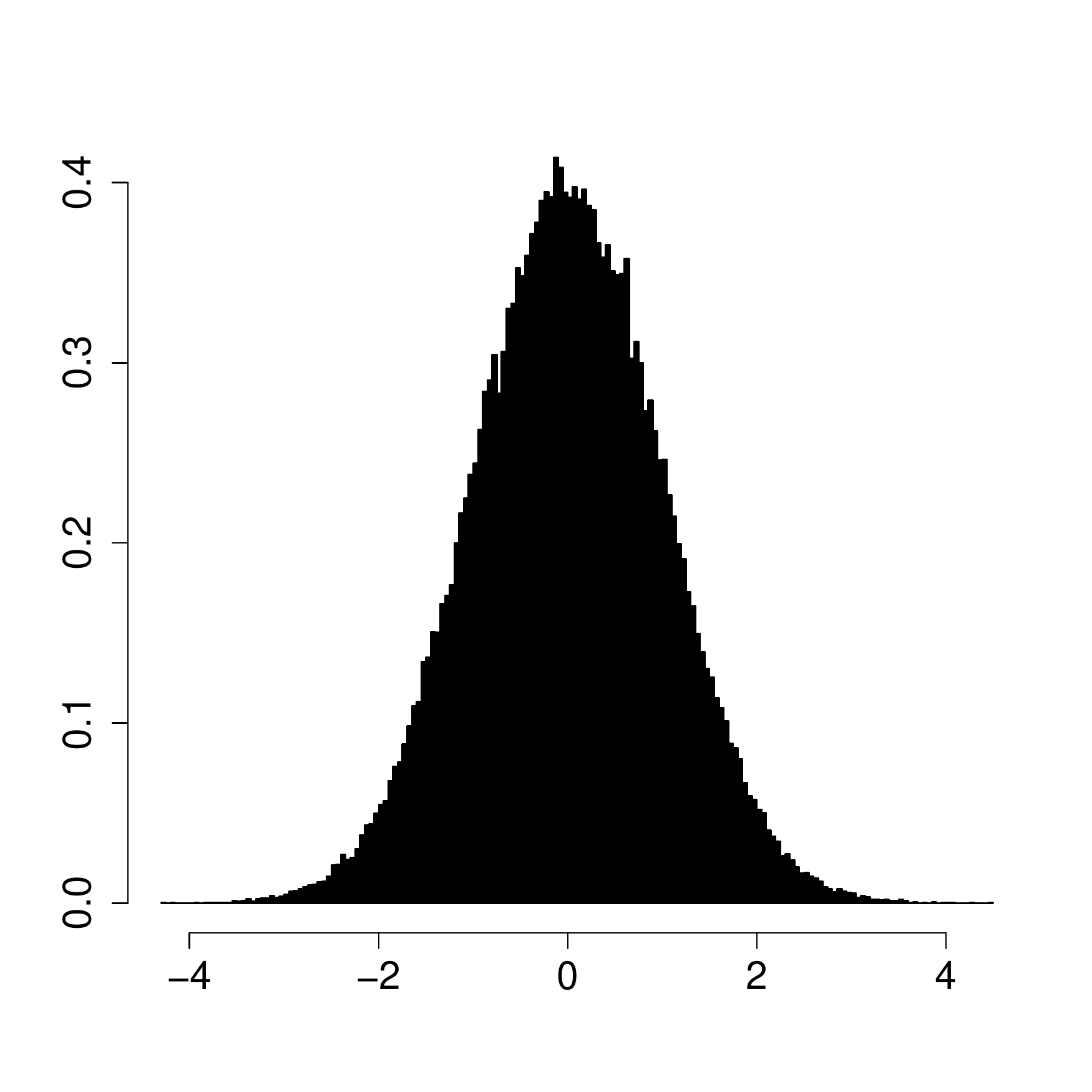}
\includegraphics[width=0.32\textwidth]{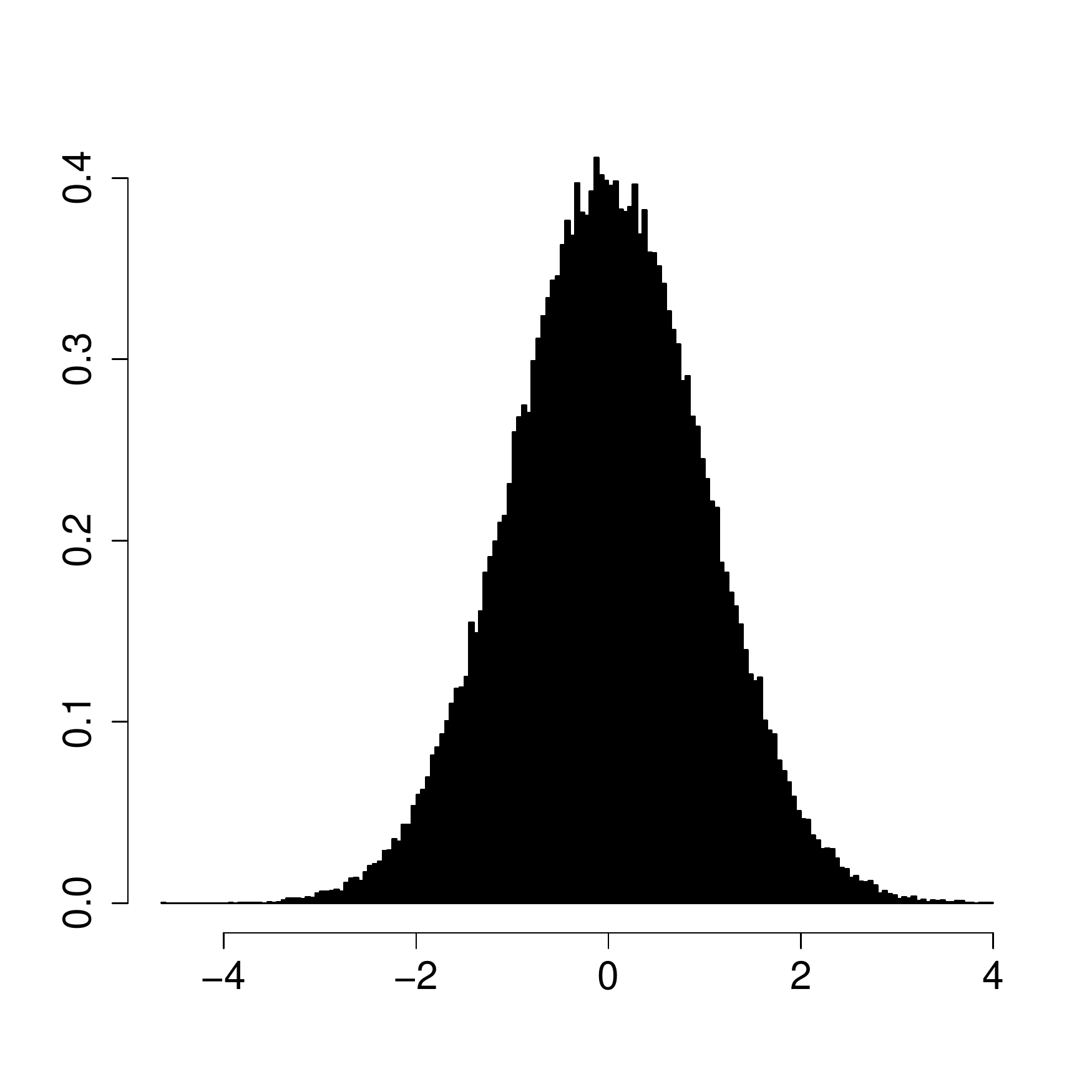}
\includegraphics[width=0.32\textwidth]{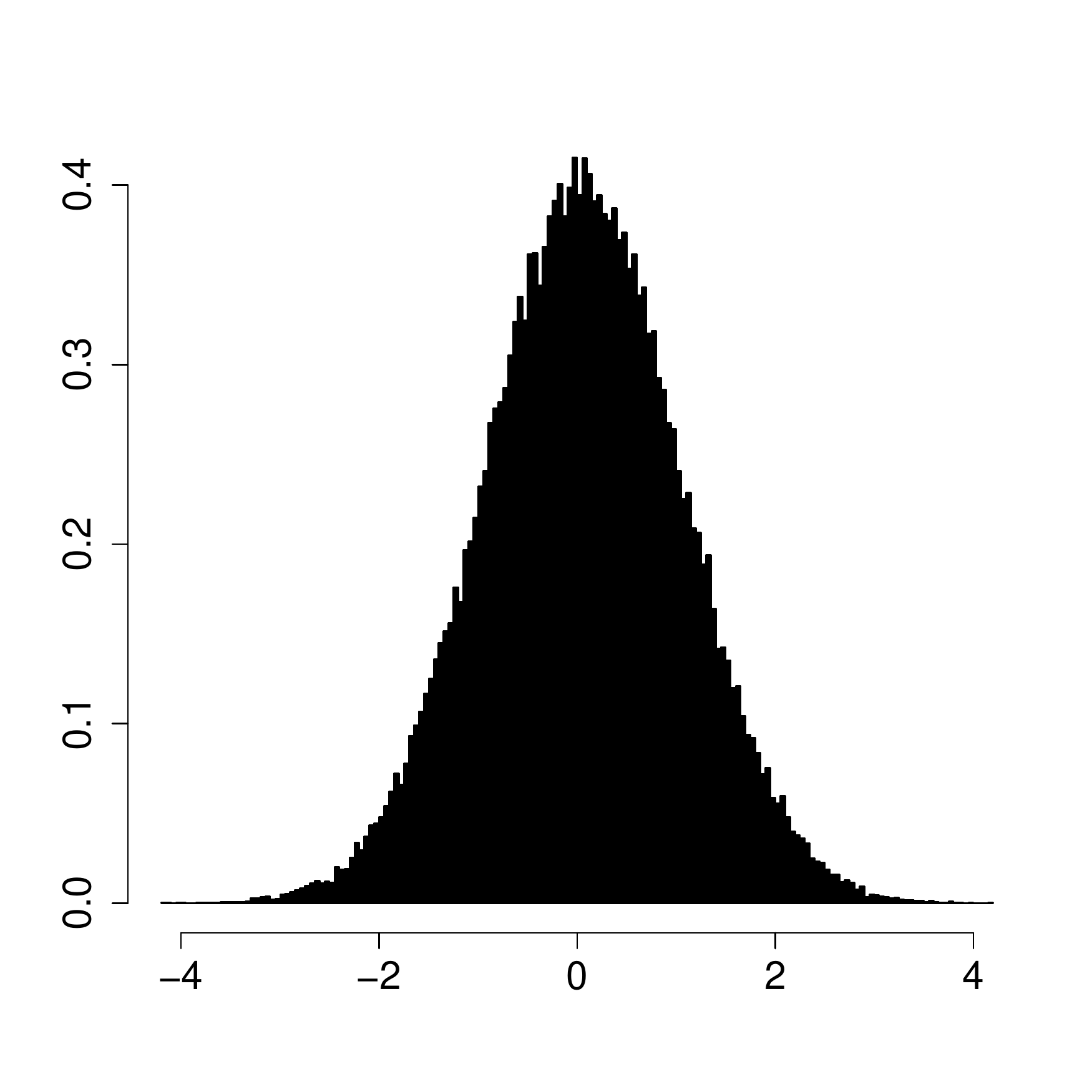}
\caption{Simulation
by Algorithm \ref{algQg}
of a population evolving according to the kernel quadratic stochastic operator
$\mathbf{Q}(\cdot,\cdot)$ with $q(x,y,z)$ given by Eq. \eqref{eqSimQ}. We assumed
$K=10000$ and present histograms for iterations
$n=1, 100, 500$ (left to right).
Top row: initial population drawn from the exponential distribution with rate $1$,
bottom row: initial population drawn from the standard normal distribution (fixed point of $\mathbf{Q}$).
We can see the mean preserving property, the sample averages are from left to right,
top row: $1.001$, $1.032$, $1.156$; bottom row: $0.004$, $-0.009$ and $0.057$.
Simulations done in R \citep{R}.
}
\label{figSimDensAlgQg}
\end{center}
\end{figure}
\subsection{Drawing from
$\mathbb{Q}_{G}^{n}(F)$} \label{secAlgFG}
One of the problems of drawing from the laws of iterates of quadratic
stochastic operators is that one needs to calculate the distribution
function of the $n$--th iterate. From our experience calculating this directly will result in
numerical errors and lengthy calculations times.
However the proof of Theorem \ref{thmWC} gives an immediate procedure to draw values from the law
of a centred kernel QSO, $\mathbb{Q}_{G}^{n}(F)$. Namely if we want a value distributed according to the law
of $\mathbb{Q}_{G}^{n}(F)$ we just need to draw an appropriate (exponential in terms of $n$)
amount of independent random variables distributed according to the laws of $F$ and $G$. We
describe this in Algorithm \ref{algQFGnaive}.
\begin{algorithm}[!ht]
\caption{Na\"ive drawing from
$\mathbb{Q}_{G}^{n}(F)$}\label{algQFGnaive}
\begin{algorithmic}
\STATE Draw $2^{n}$ random values from the law of $F$ and denote this set
$\{X_{1},\ldots,X_{2^{n}}\}$.
\STATE $X^{\text{\circledchar{$n$}}}=\frac{1}{2^{n}}(X_{1}+\ldots+X_{2^{n}} )$
\FOR{$j=0$ to $n-1$}
\STATE Draw $2^{j}$ random values from the law of $G$ and denote this set
$\{Y_{0},\ldots,Y_{2^{j}}\}$.
\STATE $U_{j} := \frac{1}{2^{j}}\left(Y_{0}+\ldots+Y_{2^{j}-1} \right)$.
\ENDFOR
\STATE \textbf{return} $X^{\text{\circledchar{$n$}}} + \sum\limits_{j=0}^{n-1}U_{j}$
\end{algorithmic}
\end{algorithm}
However the procedure described in Algorithm \ref{algQFGnaive} is na\"ive in the sense
that it is exponential in terms of the iteration number. Notice that
in the proof of Theorem \ref{thmWC} we have that if we assume that $F$ has
a finite variance, equalling $v_{F}$,
$X^{\text{\circledchar{$n$}}}$ tends almost surely to $m$ and $U_{j}$ tends
almost surely to $0$. We should expect these convergences to be rather
fast as both random variables are a sum of an exponential number of
i.i.d. random variables. We now write

$$
H^{\text{\circledchar{$n$}}} =
m + \sum\limits_{j=0}^{N-1}U_{j} + \epsilon_{N,n}
\equiv
m + Y^{\text{\circledchar{$N$}}} + \epsilon_{N,n},
$$
where $N\le n$ and $\epsilon_{N,n}$ is the deviation of $X^{\text{\circledchar{$n$}}}$
from $m$ and the tail of $U_{j}$s. Using Chebyshev we can control the probability that
both terms will not exceed a certain value. Remembering that  
$v_{G}$ is the variance associated with $G$ we have

$$
P(\vert X^{\text{\circledchar{$n$}}} - m \vert > \delta/2) \le \frac{4v_{F}}{\delta^{2}}2^{-n}
$$
and

$$
P(\vert \sum\limits_{j=N}^{n-1}U_{j} \vert > \delta/2) \le \frac{4v_{G}}{2\delta^{2}}2^{-N}(1-2^{-n}).
$$
Now obviously

$$
P(\vert X^{\text{\circledchar{$n$}}} - m \vert > \delta/2) \le P(\vert X^{\text{\circledchar{$N$}}} - m \vert > \delta/2) \le \frac{4v_{F}}{\delta^{2}}2^{-N}
$$
so for a given $\delta$ we can choose $N$ large enough so that the probability, $\alpha$, of drawing
a value which is ``off the correct distribution'' by more than $\delta$ is as small as we desire.
Namely we have

$$
N(\alpha,\delta,n) \le \log\frac{4\max(v_{F},v_{G}(1-2^{-n})/2)}{\delta^{2}\alpha}
=
\max(\log\frac{4v_{F}}{\delta^{2}\alpha},\log\frac{4v_{G}(1-2^{-n})}{\delta^{2}\alpha})
$$
and if we are interested in drawing from $G^{\text{\circledchar{$\infty$}}}$

$$
N(\alpha,\delta,\infty) \le \log\frac{4\max(v_{F},v_{G}/2)}{\delta^{2}\alpha}.
$$
We describe this modification in Algorithm \ref{algQFGapprox}. 
\begin{algorithm}[!ht]
\caption{Approximate drawing from
$\mathbb{Q}_{G}^{n}(F)$}\label{algQFGapprox}
\begin{algorithmic}
\STATE $N(\alpha,\delta,n) := \log\frac{4\max(v_{F},v_{G}(1-2^{-n})/2)}{\delta^{2}\alpha}+1$
\FOR{$j=0$ to $N-1$}
\STATE Draw $2^{j}$ random values from the law of $G$ and denote this set
$\{Y_{0},\ldots,Y_{2^{j}-1}\}$.
\STATE $U_{j} := \frac{1}{2^{j}}\left(Y_{0}+\ldots+Y_{2^{j}-1} \right)$.
\ENDFOR
\STATE \textbf{return} $m + \sum\limits_{j=0}^{N-1}U_{j}$
\end{algorithmic}
\end{algorithm}
It still remains an open question whether drawing $2^{N}$ values will be feasible.
If we want to draw a population of $K$ individuals then it does not suffice
to choose a small $\alpha$ independently of $K$. This is akin to the multiple testing
problem --- with $K$ large enough just by chance we will observe an event of probability $\alpha$.
Therefore one way is a ``Bonferroni'' correction --- if on the individual level
we want an error with probability $\alpha$ then in Algorithm \ref{algQFGapprox} we need
to take $\alpha/K$ instead of $\alpha$.
We illustrate Algorithm \ref{algQFGapprox} by populations
using the same kernel and initial distributions as in Fig. \ref{figSimDensAlgQg}.
We present the populations' histograms in Fig. \ref{figSimDensAlgFG}.
For the approximate algorithm we took $\alpha=0.05$ and $\delta=0.01$.
This resulted in $N=14$. Each population is of size $K=10000$ and
without the ``Bonferroni'' style correction sampling of all
individuals was instantaneous (about $30$s for each on 
a $1.4$GHz AMD Opteron Processor 6274 running Ubuntu 12.04 node of a computational cluster).
The crucial tuning parameter is $\delta$. If we choose $\delta=0.001$
with the same $\alpha=0.05$ then $N$ rose to $19$ and the sampling
of the whole population became intolerable. 
The Bonferroni correction increases $N$ to $23$ and $28$ respectively for
the two values of $\delta$. 
About $4$ hours were needed to simulate the $K=10000$ individuals population (with $N=23$)
on the same node.
However if we compare the histograms of Figs. \ref{figSimDensAlgQg}, 
\ref{figSimDensAlgFG} and \ref{figSimDensAlgFGBonf} we see 
that there is no need for the correction.
In fact the deviations of the mean from its correct value
($1$ and $0$ with the two different seed distributions) 
are similar in both cases and smaller than in the case of
Algorithm \ref{algQg} for $n=500$.
\begin{figure}[!ht]
\begin{center}
\includegraphics[width=0.32\textwidth]{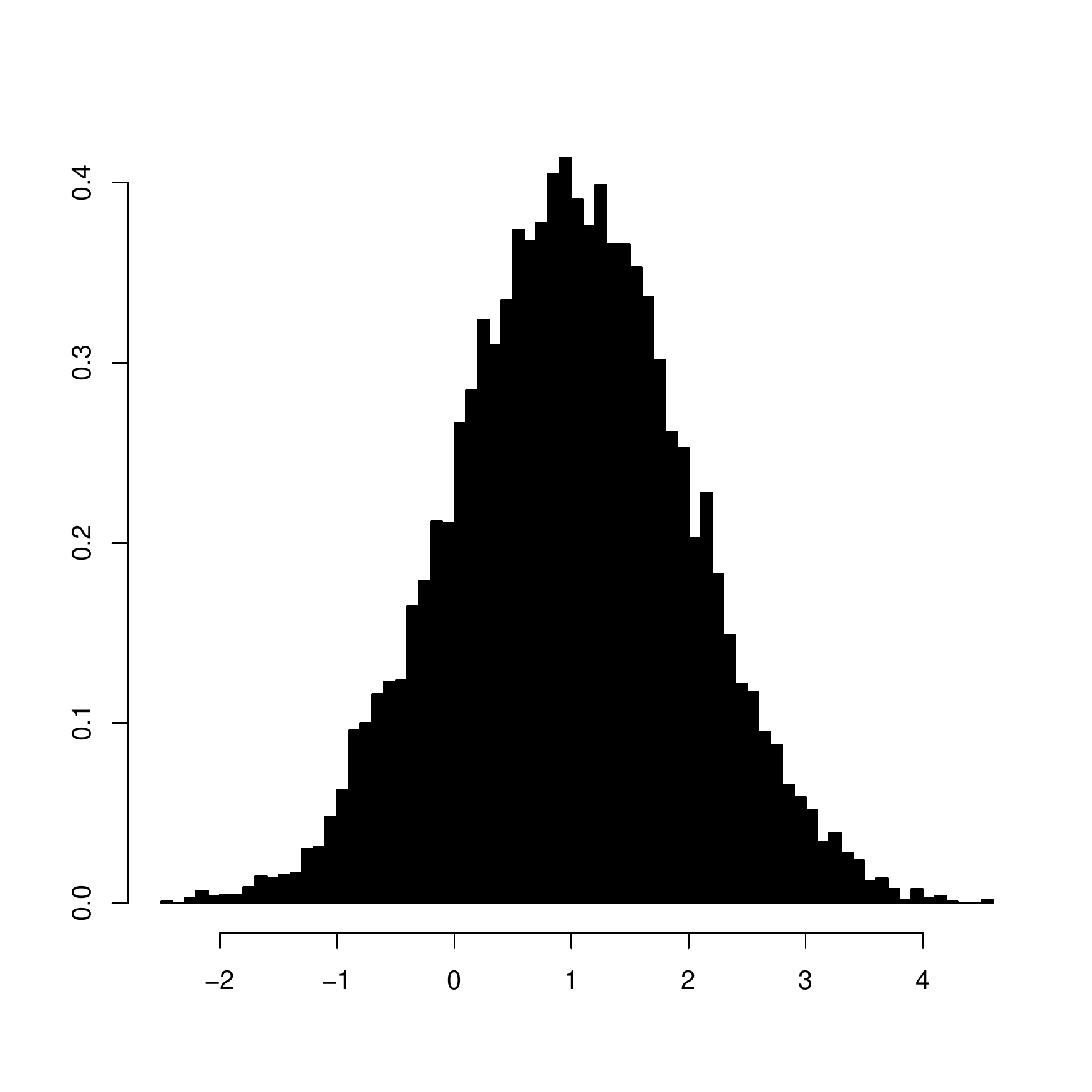}
\includegraphics[width=0.32\textwidth]{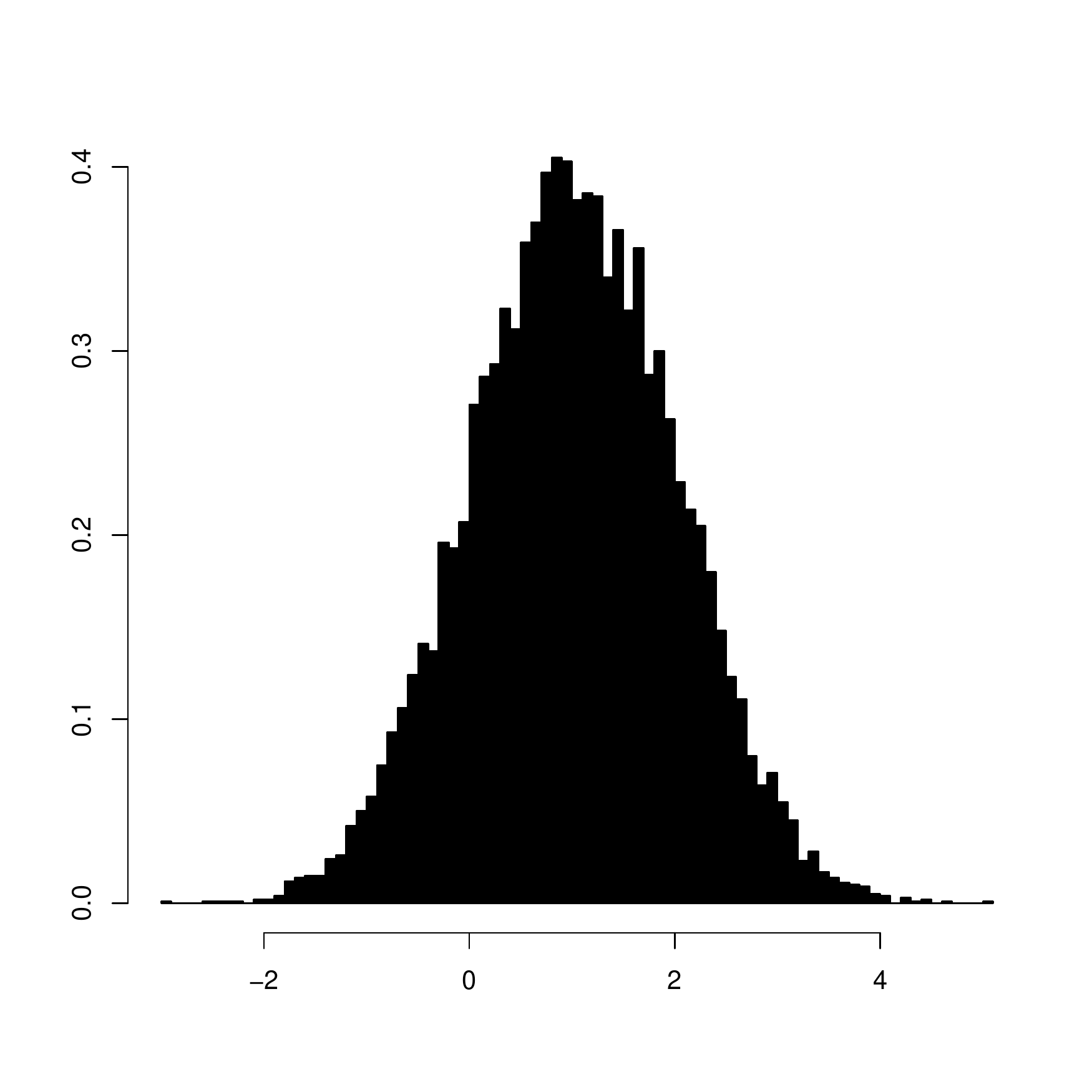}
\includegraphics[width=0.32\textwidth]{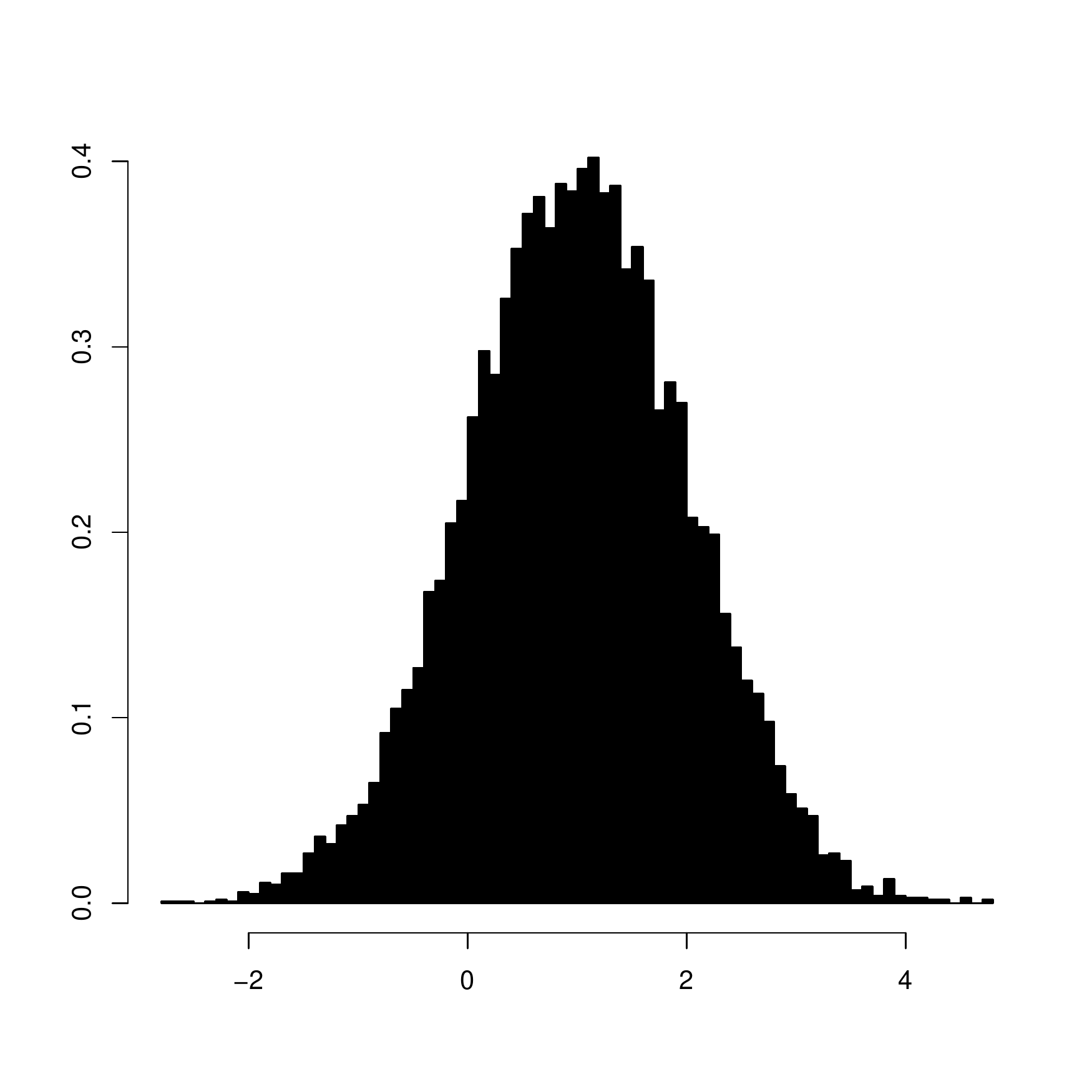} \\
\includegraphics[width=0.32\textwidth]{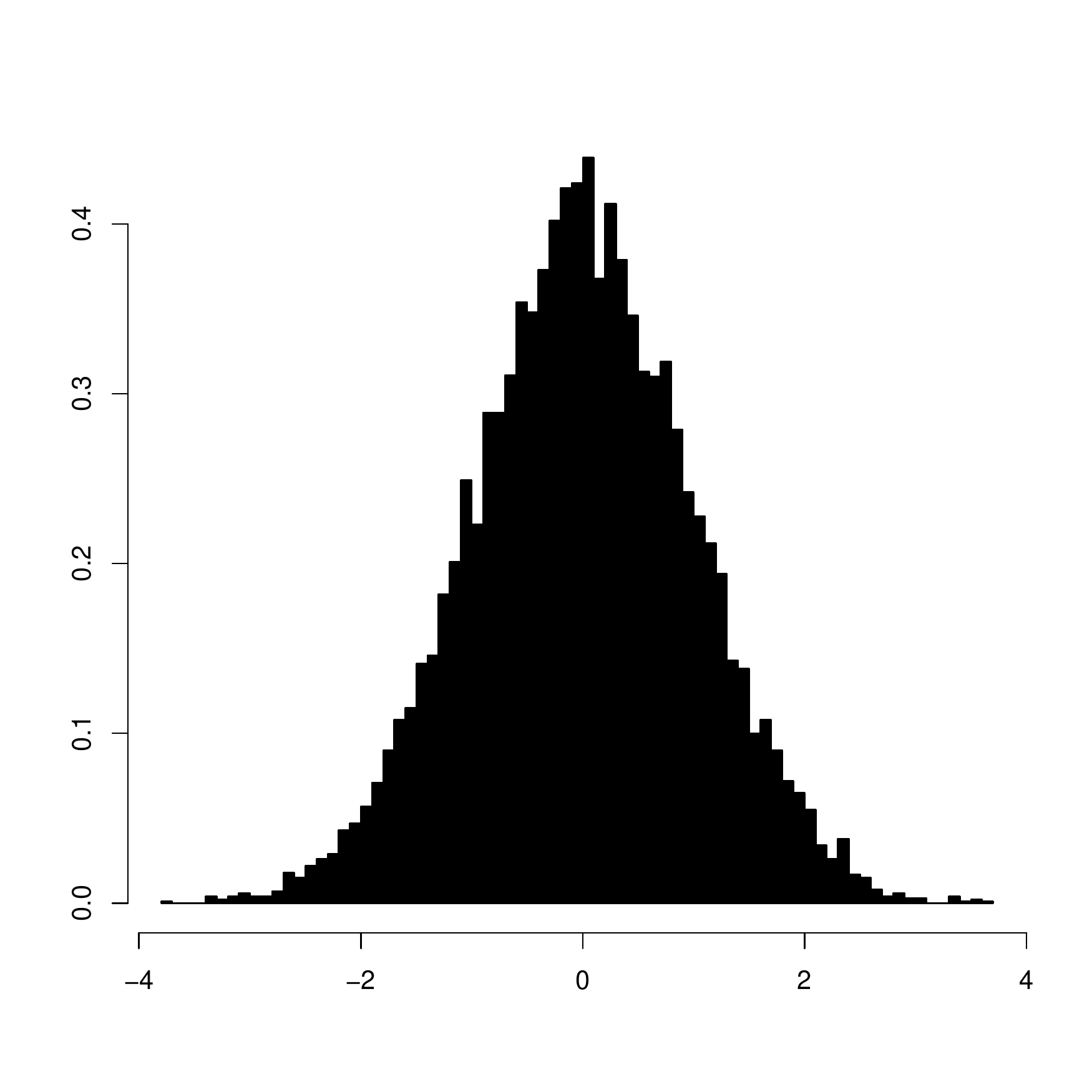}
\includegraphics[width=0.32\textwidth]{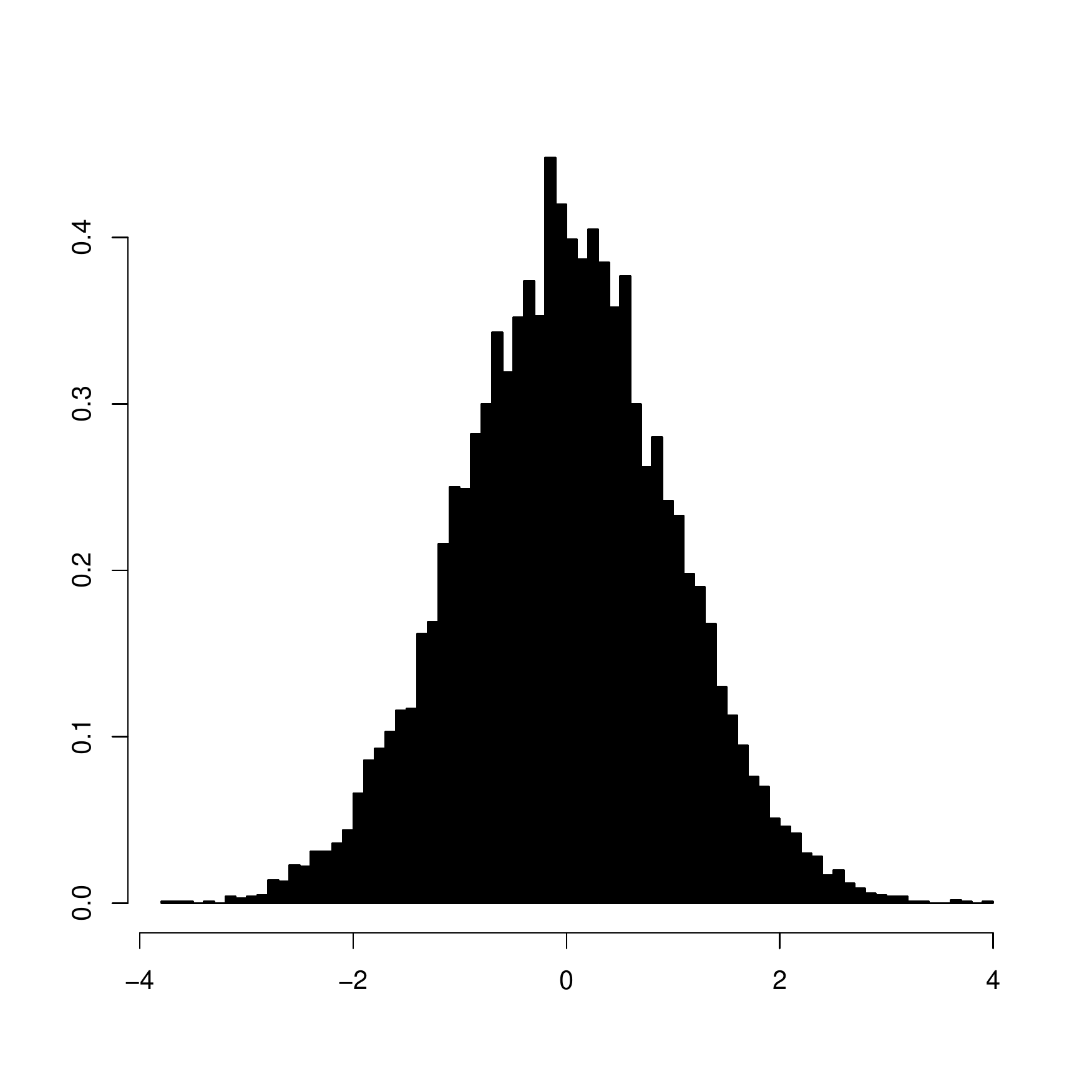}
\includegraphics[width=0.32\textwidth]{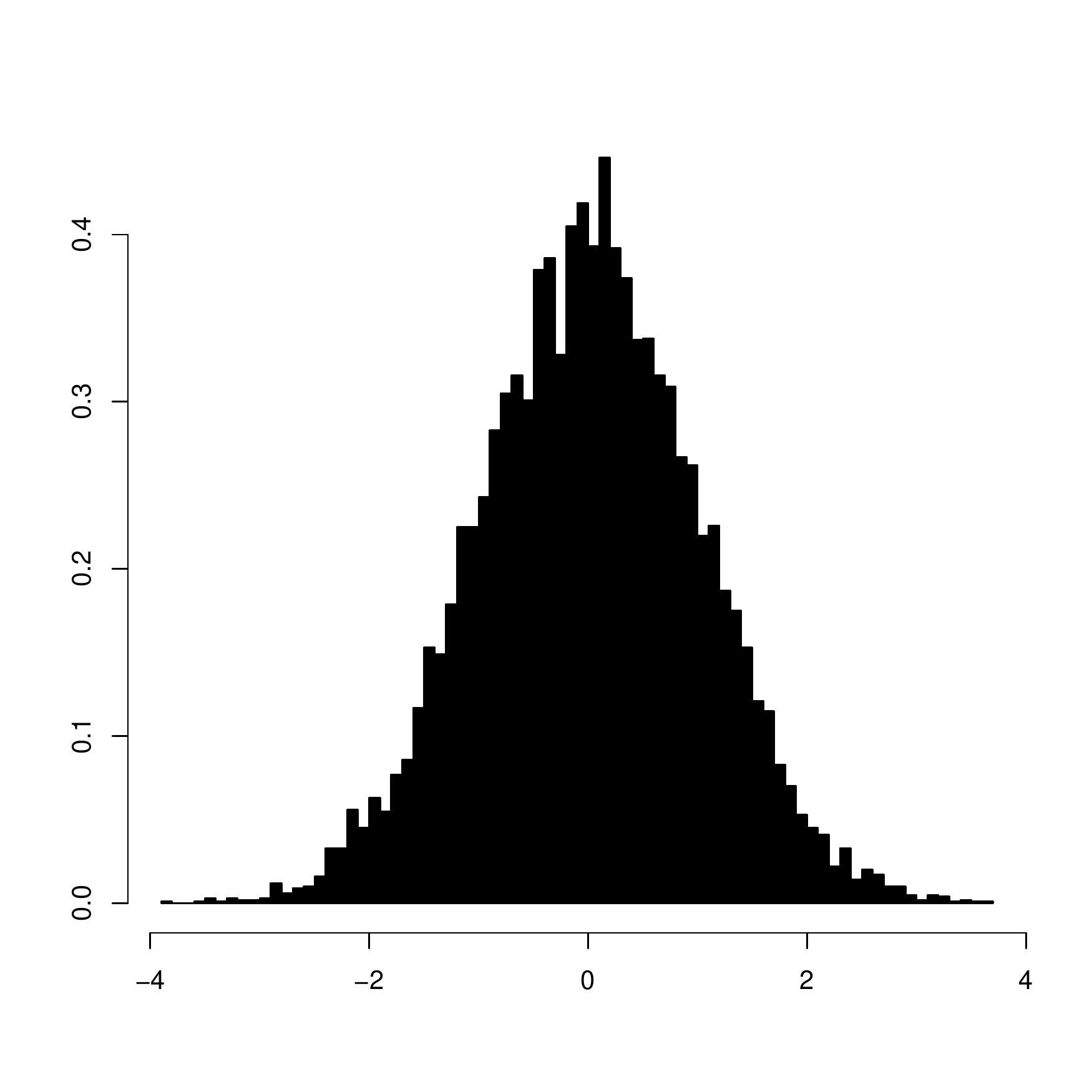}
\caption{Simulation
by Algorithm \ref{algQFGapprox}
of the law of kernel quadratic stochastic operator
$\mathbf{Q}^{n}(\cdot,\cdot)$ with $q(x,y,z)$ given by Eq. \eqref{eqSimQ}. We assumed
$K=10000$ and present histograms for iterations
$n=1, 100, 500$ (left to right). 
Top row: initial population drawn from the exponential distribution with rate $1$,
bottom row: initial population drawn from the standard normal distribution (fixed point of $\mathbf{Q}$).
We can see the mean preserving property, the sample averages are from left to right,
top row: $0.989$, $1.018$, $0.999$; bottom row: $-0.004$, $-0.013$ and $0.014$.
The top left graph is of course completely wrong as the approximate
algorithm has no knowledge of the initial exponential distribution, it
uses only its mean and variance.
Simulations done in R.
}
\label{figSimDensAlgFG}
\end{center}
\end{figure}
\begin{figure}[!ht]
\begin{center}
\includegraphics[width=0.32\textwidth]{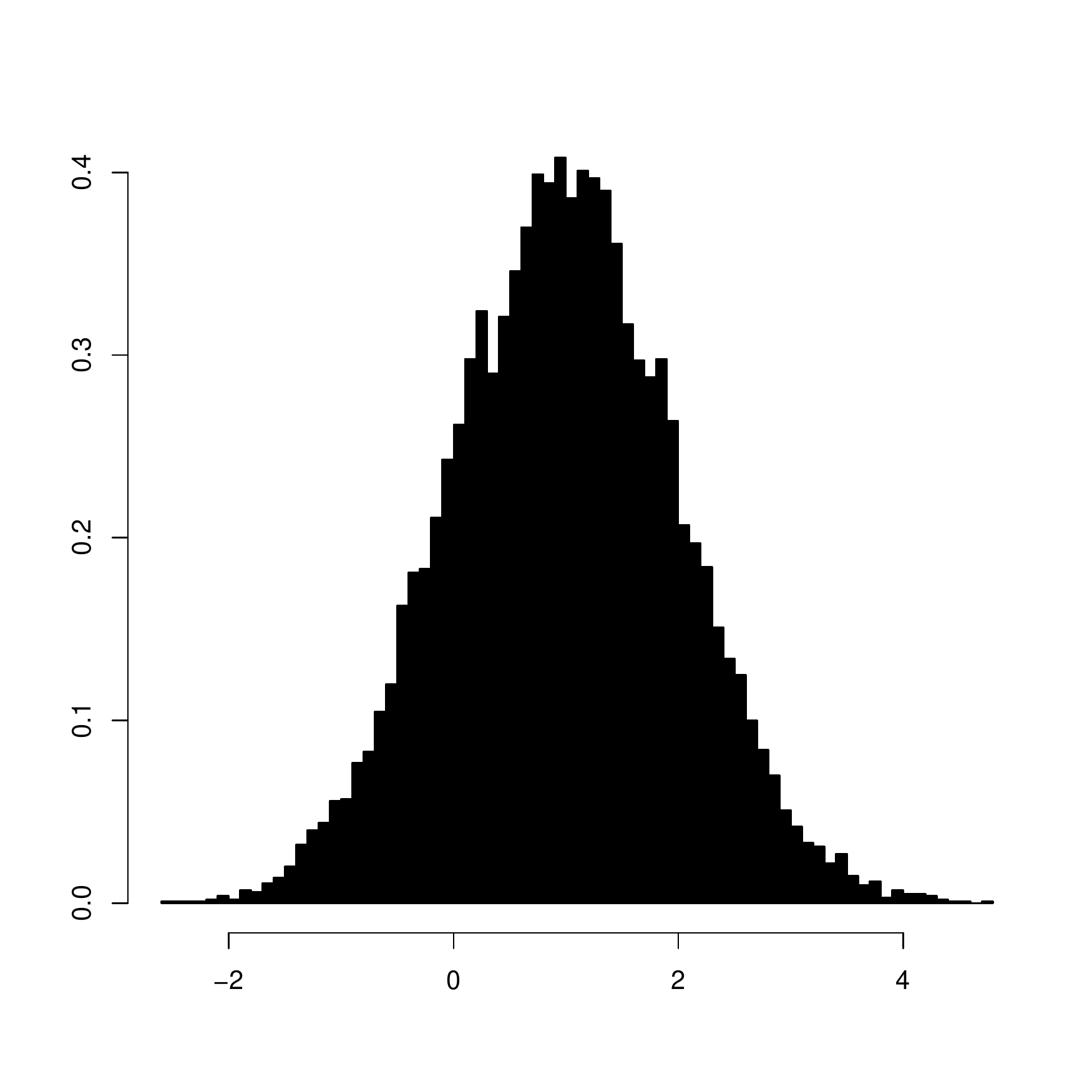}
\includegraphics[width=0.32\textwidth]{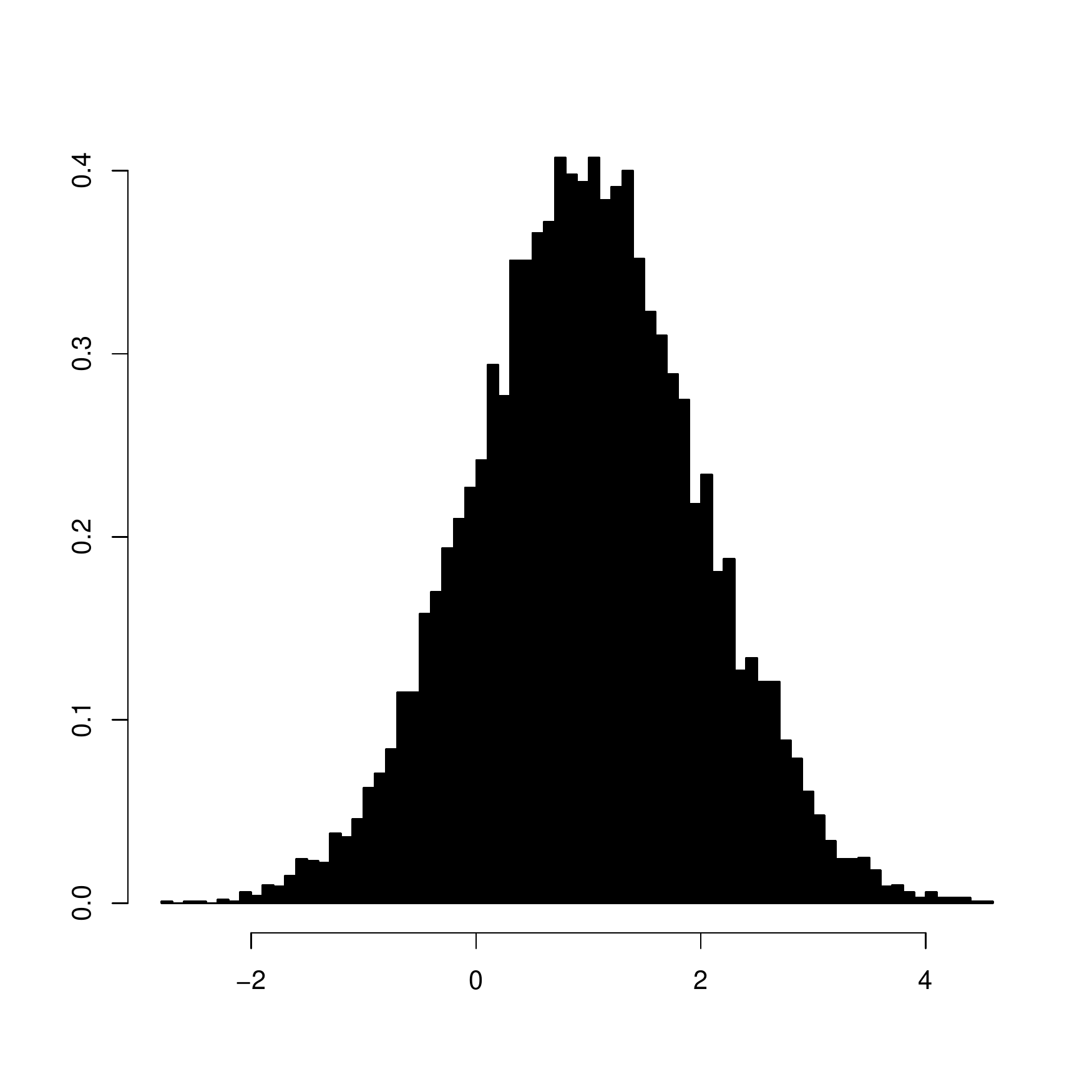}
\includegraphics[width=0.32\textwidth]{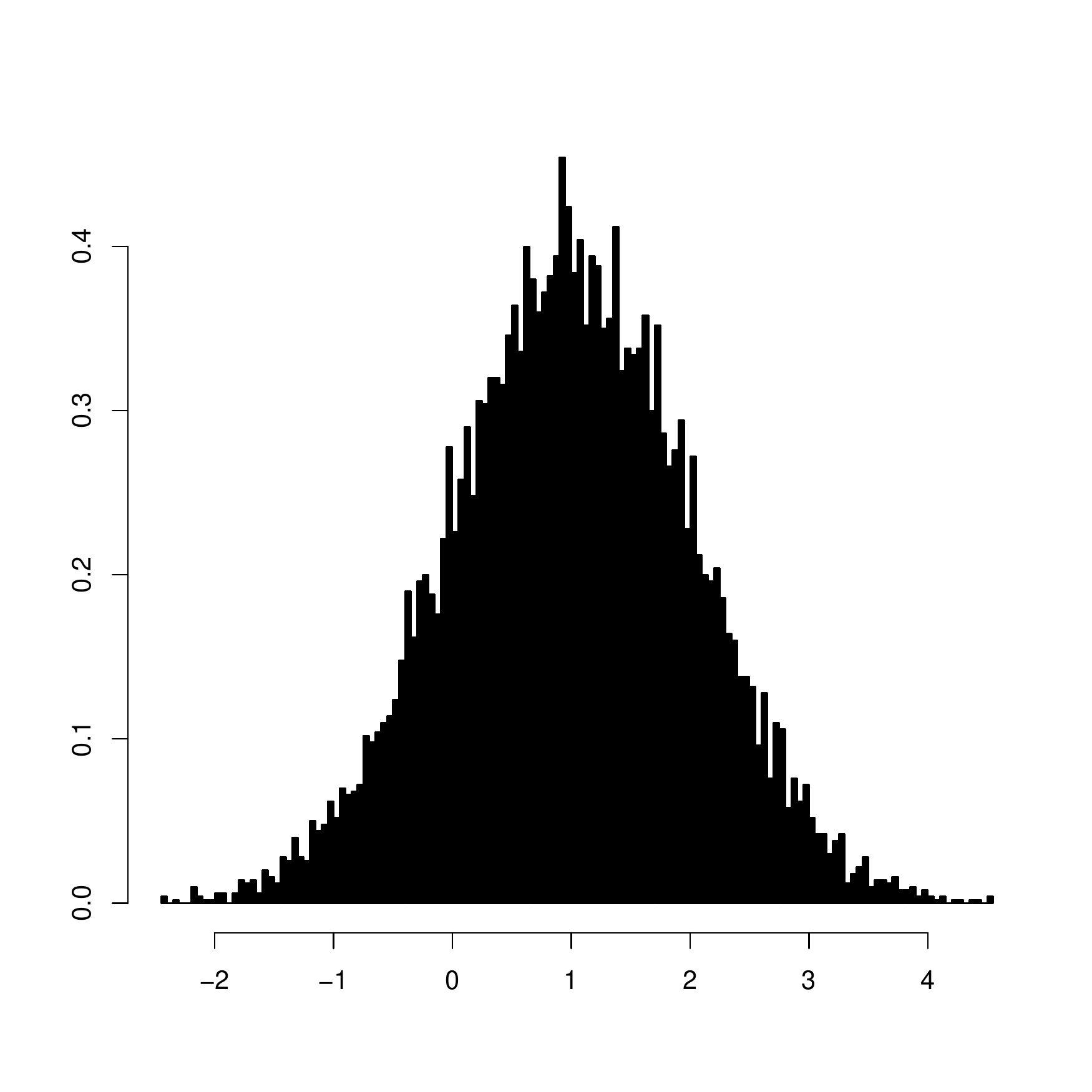} \\
\includegraphics[width=0.32\textwidth]{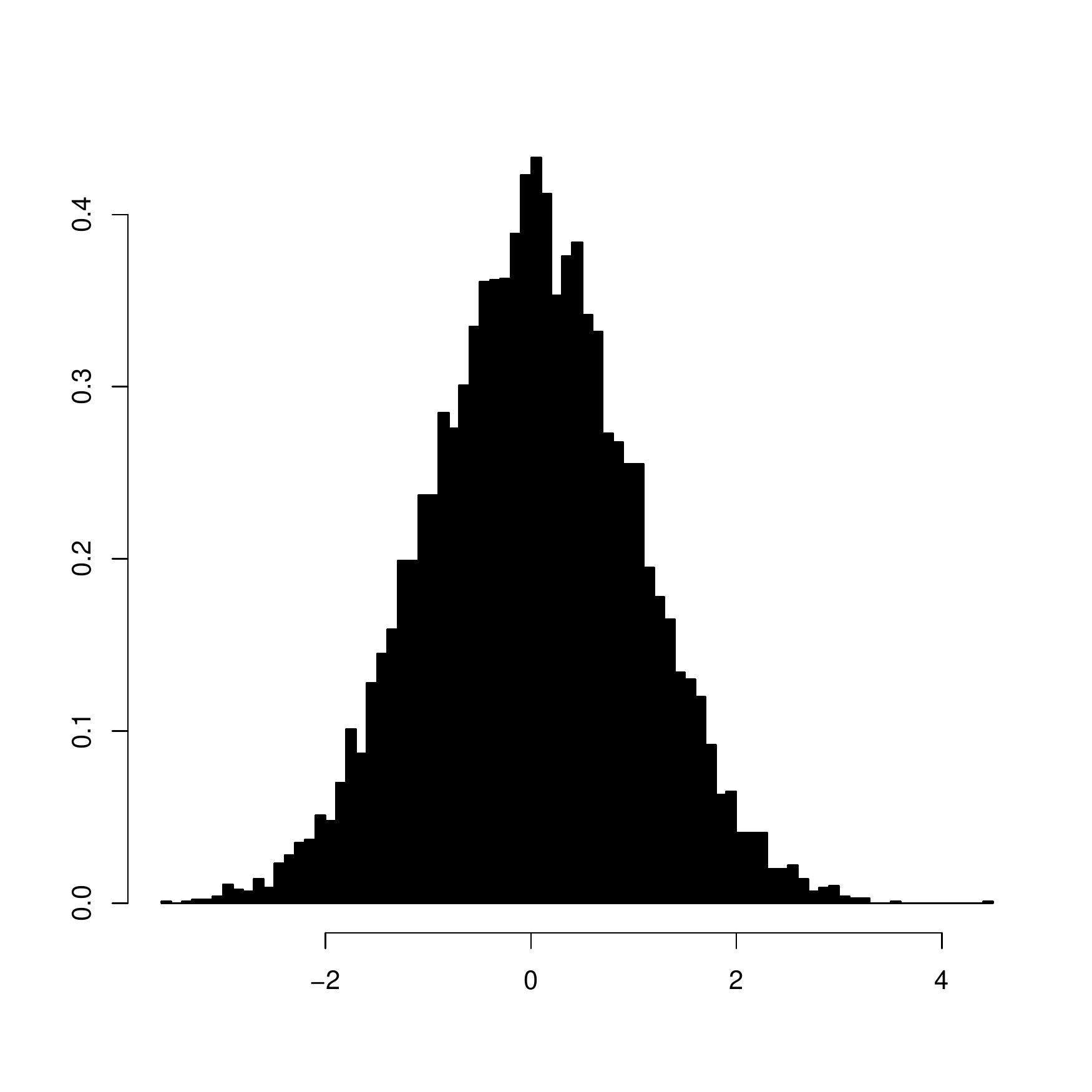}
\includegraphics[width=0.32\textwidth]{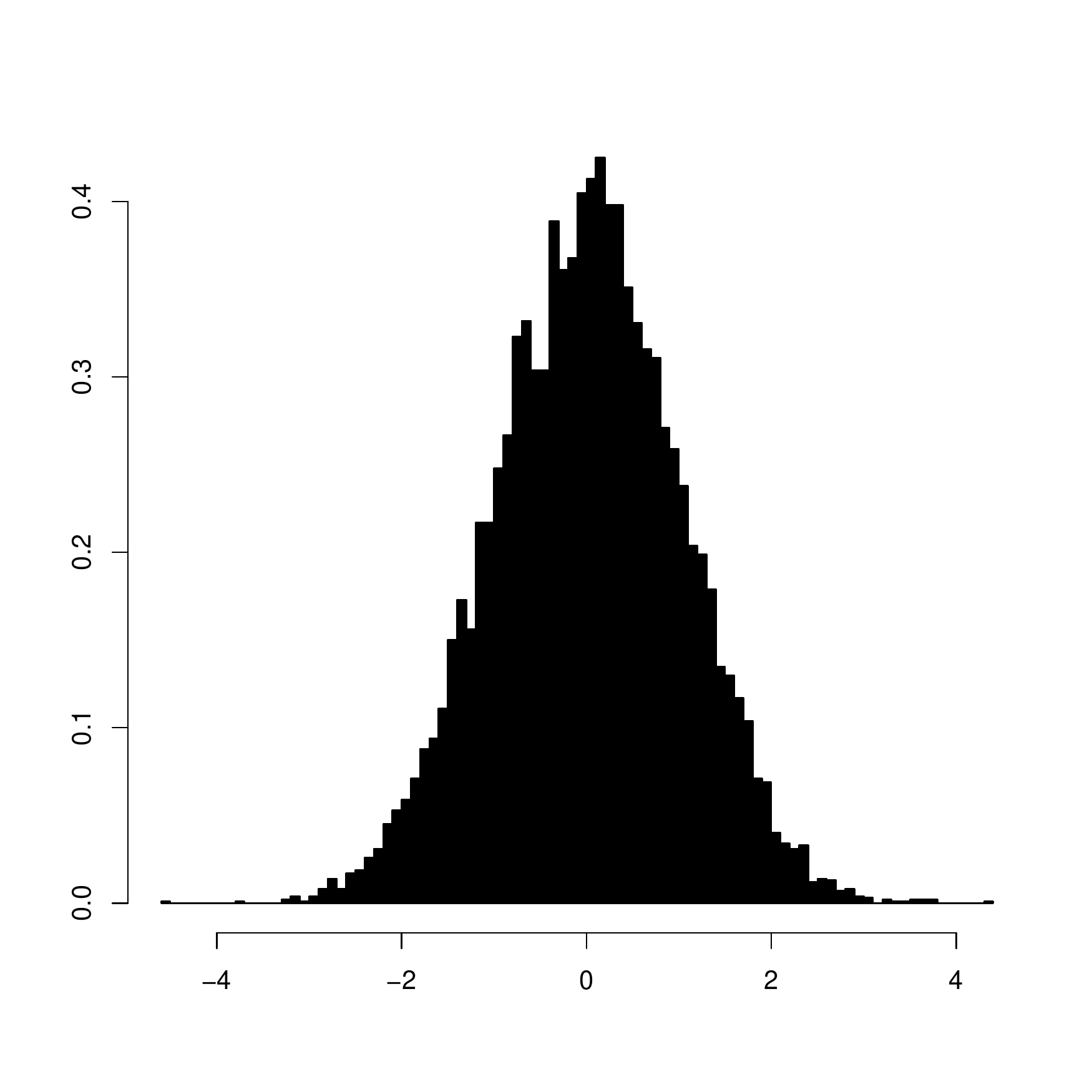}
\includegraphics[width=0.32\textwidth]{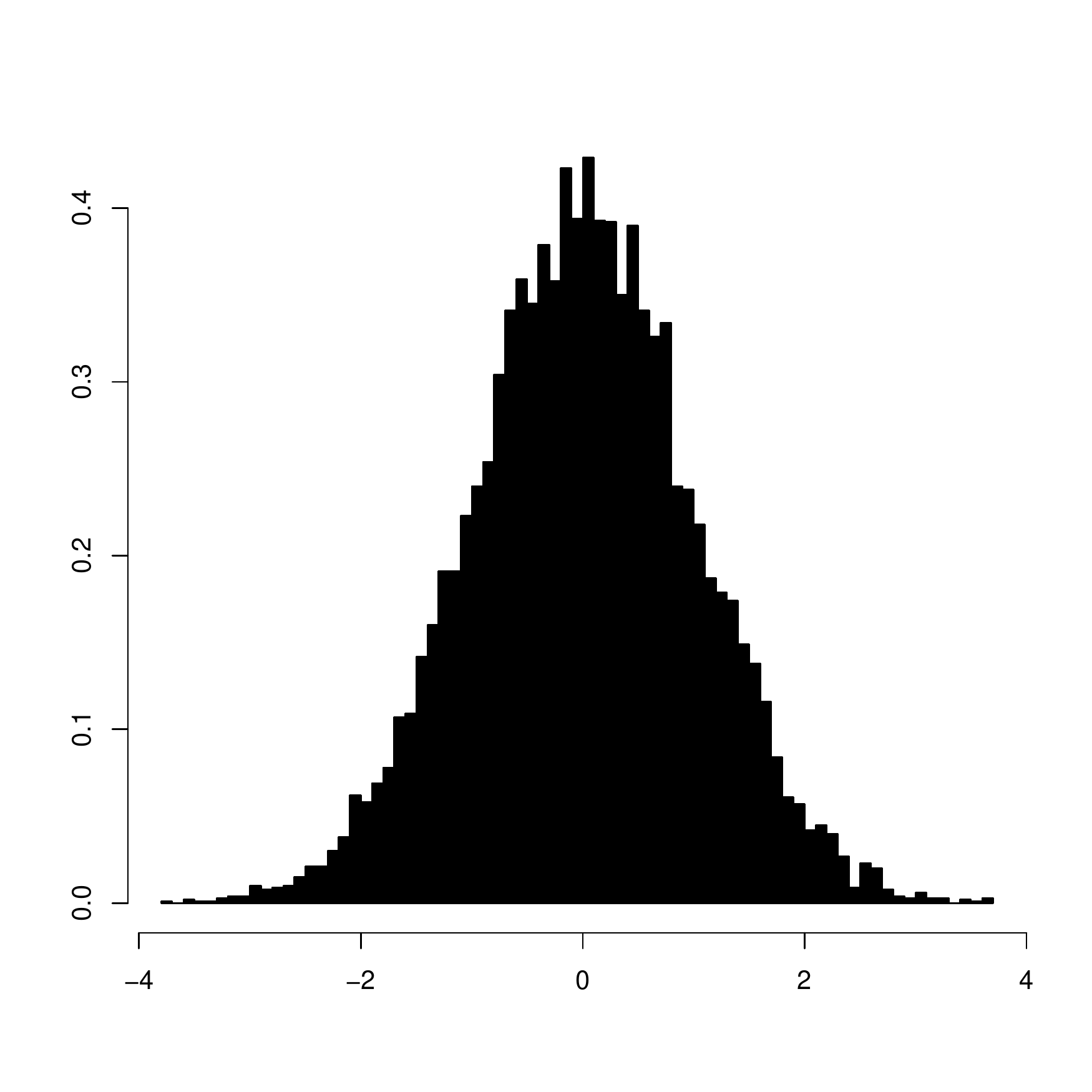}
\caption{Simulation
by Algorithm \ref{algQFGapprox}
of the law of kernel quadratic stochastic operator
$\mathbf{Q}^{n}(\cdot,\cdot)$ with $q(x,y,z)$ given by Eq. \eqref{eqSimQ}
with the ``Bonferroni'' type correction. We assumed
$K=10000$ and present histograms for iterations
$n=1, 100, 500$ (left to right). 
Top row: initial population drawn from the exponential distribution with rate $1$,
bottom row: initial population drawn from the standard normal distribution (fixed point of $\mathbf{Q}$).
We can see the mean preserving property, the sample averages are from left to right,
top row: $0.980$, $0.979$, $1.003$; bottom row: $0.010$, $0.017$ and $0.007$.
The top left graph is of course completely wrong as the approximate
algorithm has no knowledge of the initial exponential distribution, it
uses only its mean and variance.
Simulations done in R.
}
\label{figSimDensAlgFGBonf}
\end{center}
\end{figure}
\subsection{Comparing both algorithms}
Both Algorithms \ref{algQg} and \ref{algQFGapprox}  have their advantages and disadvantages.
The main advantage of the latter is speed (provided $N$ does not need to be overly large).
One just has to simulate values from a univariate $G$. This is
as we approximate the seed distribution by its mean value --- indicating
that this will work only when $n$ is large, i.e. many iterations have passed
and only information on the expectation remains. On the other hand
we can control $N$ very precisely as we know $F$ and $G$.
In Section \ref{secAlgFG} we used Chebyshev but for a specific pair
of distributions a much better bound will be certainly available.
We can expect rapid convergence in distribution
as iterations of the operator cause an exponential growth of the number
of ``independent components'' describing the law of $\mathbb{Q}_{G}^{n}(F)$.
However if $N$ is too large to be practical one can always use a smaller
manageable value but then of course the error probabilities will increase.
Algorithm \ref{algQg} allows one to simulate a whole population evolving.
This is an advantage if one wants to visualize the evolution. On the other
hand and if one is just interested in the law of $\mathbb{Q}_{G}^{n}(F)$
or $\mathbb{Q}_{G}^{\infty}(F)$ then the need to simulate a whole history
can be overly lengthy. This algorithm does not require the drawing of a large
number of random variables but has another problem which as we saw in the example simulation
caused larger and larger deviations from the true distribution. In a computer simulation
we cannot have an infinite population size --- only a finite number of individuals.
This means that after iterations of mixing more and more dependencies will be appearing
in the population --- something which our theory at the moment does not account for.
In fact when we look at the simulation results presented in Fig. \ref{figSimDensAlgQg}
we can see that the population average is slowly deviating from $0$ ---
a consequence of the dependencies due to finite sample size.
One can actually think that all of the above issues, especially the exponential
number of ``independent components'' illustrate or rather characterize the
complexity of the structure of quadratic stochastic operators. Fortunately
one can start quantifying this complexity as we did with the Chebyshev
bound. All simulations actually indicate rapid convergence for ``decent'' $F$ and $G$
distributions. On the other hand we restricted ourselves to a very specific class ---
centred kernel quadratic stochastic operators.
In the full set of quadratic stochastic operators we should expect
many more interesting dynamics.

\section*{Acknowledgments}
We grateful to Wojciech Bartoszek for many helpful discussions, comments and insights.
Krzysztof Bartoszek was supported by
Svenska Institutets \"Ostersj\"osamarbete scholarship nrs. 00507/2012, 11142/2013, 19826/2014.

\bibliographystyle{plainnat}
\bibliography{KBartoszekPulka}

\end{document}